\documentclass[3p, authoryear, 10pt]{elsarticle}
\journal{International Journal of Approximate Reasoning}

\usepackage{amssymb}
\usepackage{amsmath}
\usepackage[colorlinks]{hyperref}
\usepackage{cleveref}
\usepackage{xcolor}
\usepackage{comment}
\usepackage{pgfplots}
\usepackage{subcaption}

\definecolor{salmon}{RGB}{250,128,114}
\definecolor{apricot}{RGB}{128,128,0}
\definecolor{apricot}{RGB}{230,170,110}
\definecolor{pastelblue}{RGB}{80,120,180}
\pgfplotsset{compat=1.18}

\newcommand{\D}{\mathcal D}
\newcommand{\dd}{\mathrm{d}}
\newcommand{\E}{\mathbb E}
\newcommand{\Ecb}{\mathcal E^{\mathrm{cb}}}
\newcommand{\Ecbinf}{\mathcal E^{\mathrm{cb},\infty}}

\newcommand{\EcbT}{\mathcal E^{\mathrm{cb},T}}
\newcommand{\Ee}{\mathcal E}

\newcommand{\EH}{\mathcal E^{\mathrm{Hoeff}}}
\newcommand{\EHa}{E^{\mathrm{H}}_\alpha}
\newcommand{\F}{\mathcal F}

\newcommand{\Hh}{\mathcal H}

\newcommand{\Nn}{\mathbb N}
\newcommand{\one}[1]{\mathbf{1}_{#1}}

\newcommand{\Pp}{\mathbb P}
\newcommand{\PR}{\mathcal P}

\newcommand{\R}{\mathbb R}
\newcommand{\supp}{\mathrm{Supp}}
\newcommand{\T}{\mathcal T}
\newcommand{\X}{\mathcal X}
\newcommand{\Y}{\mathcal Y}
\newcommand{\Z}{\mathcal Z}

\newtheorem{definition}{Definition}
\newtheorem{lemma}{Lemma}
\newtheorem{proposition}{Proposition}
\newtheorem{theorem}{Theorem}

\newproof{proof}{Proof}

\begin{document}

\bibliographystyle{elsarticle-harv} 

\begin{frontmatter}
    \title{On the optimality of coin-betting for mean estimation} 
    \author{Eugenio Clerico}
    \ead{eugenio.clerico@gmail.com}
    \address{Universitat Pompeu Fabra, Barcelona, Spain}
    \begin{abstract}  
        We consider the problem of testing the mean of a bounded real random variable. We introduce a notion of optimal classes for e-variables and e-processes, and establish the optimality of the coin-betting formulation among e-variable-based algorithmic frameworks for testing and estimating the (conditional) mean. As a consequence, we provide a direct and explicit characterisation of all valid e-variables and e-processes for this testing problem. In the language of classical statistical decision theory, we fully describe the set of all admissible e-variables and e-processes, and identify the corresponding minimal complete class.
    \end{abstract}
\end{frontmatter}

\section{Introduction}\label{sec:intro}
Estimating the mean of a random variable from empirical observations is a classical problem in statistics. To account for uncertainty, a widely used approach consists in constructing a confidence set, known to contain the true mean with high probability, rather than relying solely on a point estimate. When the data are observed sequentially, one might want to update this set as new data-points become available. However, such procedure may compromise the validity of the statistical guarantee, if this was designed for a fixed sample size. To address this issue, \citet{darling67confidence} introduced the concept of \emph{confidence sequence}, a data-adaptive sequence of confidence sets whose intersection contains the desired mean with high probability.

\cite{orabona2023tight} and \cite{waudbysmith23estimating} have recently explored \emph{algorithmic} approaches that yield some of the tightest confidence sequences for  the mean of a bounded real random variable. Both papers propose setting up a series of sequential \emph{coin-betting} games, one per each mean candidate value $\mu$, where a player sequentially bets on the difference between $\mu$ and the upcoming observation. If $\mu$ matches the true mean, the game is fair, and substantial gains unlikely. A confidence sequence is  obtained by excluding those values  $\mu$ that allowed the player to accumulate significant wealth. 

This coin-betting approach to mean estimation is a particular instance of a broader algorithmic framework for constructing confidence sequences through \emph{sequential hypothesis testing}, which can be framed in terms of betting games where at each round the player has to select an \emph{e-variable} \citep{shafer2021testing, ramdas2022admissible, ramdas2023game}. E-variables, non-negative random variables whose expectation is bounded by $1$ under the tested hypothesis \citep{grunwald2024safe}, have recently emerged as a powerful and increasingly popular tool for anytime-valid hypothesis testing. By serving as building blocks for constructing non-negative super-martingales, which  can be seen as representing the wealth of a player in a \emph{betting game}, e-variables naturally lend themselves to game-theoretic interpretations \citep{shafer2019game, ramdas2024hypothesis}. 

The main goal of this work is to illustrate and formalise that, when \emph{sequentially} testing and estimating the (conditional) mean of a bounded real random variable, no e-variable procedure yields strictly better guarantees than the coin-betting approach. In a sense to be clarified later, coin-betting is \emph{optimal}, as it represents the ``simplest'' formulation among those that cannot be strictly performed by any other such testing-by-betting approach. One main novelty of this work is the introduction of a notion of optimality at the level of \emph{sets} of e-variables. This perspective differs from much of the existing literature, primarily focused on the optimality of an individual e-variable, or wealth process, in the betting game (e.g., \emph{log-optimality} in \citealt{koolen2022log,grunwald2024safe, larsson2024numeraire}, or \emph{admissibility} in \citealt{ramdas2022admissible}). We remark that the perspective adopted in this work can be seen as an adaptation, to the setting of e-variables, of the classical statistical problem of identifying a minimal complete class of tests (see \citealp{lehmann2005testing}). Further discussion of this connection is deferred to \Cref{sec:perspectives}.

The first part of this work focuses on  round-wise \emph{testing-by-betting}, a scenario where the player iteratively picks a \emph{single-round} e-variable, whose choice may depend on the past observations. This procedure naturally applies to the setting where the observations are known to be independently drawn from a fixed probability distribution, whose mean has to be estimated. However, we remark that testing via sequential betting with single-round e-variables is not the most general form of sequential testing with e-variables, which typically relies on \emph{multi-round} e-variables and \emph{e-processes} \citep{shafer2021testing, koolen2022log, ramdas22testing, ramdas2024hypothesis}. These tools allow for testing hypotheses over the entire data sequence, including assumptions about the dependence structure (e.g., i.i.d.~or fixed conditional mean). In such cases, restricting the player to select a single-round e-variable at each step may be highly limiting \citep{koolen2022log, ramdas22testing}. The second part of this work considers this broader setting. We establish that when the sequence has fixed conditional mean, coin-betting remains optimal even among testing methods based on multi-round e-variables and e-processes. However, we also show that this optimality result no longer holds under the more restrictive hypothesis of i.i.d.~observations.

It is worth noting that an alternative way to frame the main contribution of this work is as a concrete and direct characterisation of the family of e-variables and e-processes for testing the (conditional) mean of a real bounded random variable. More precisely, the e-variables and e-processes for these tests are exactly the non-negative measurable functions or processes that are majorised by a coin-betting e-variable or e-process. Since the coin-betting formulation provides a very explicit expression for these objects, our results directly yield a fully explicit description of the full set of e-variables and e-processes for the problem at hand. Following  a first pre-print of this manuscript, general characterisations for the e-variables when testing hypotheses defined by linear constraints were established by \cite{clerico24optimal} and \cite{larsson2025evariables}. These results directly imply ours  for single-round e-variables (\Cref{thm:main}), which is also explicitly discussed by these works as an application. However, the approach we present here is direct and tailored to the simple setting considered, making it valuable for building intuition, while the more general results are significantly more abstract. Moreover, the two aforementioned works focus exclusively on single-round e-variables, whereas our contribution provides a complete and  explicit characterisation of both e-variables and e-processes in the sequential  conditional mean testing problem. For further discussion on these works, see \Cref{sec:perspectives}.

\subsection*{Notation}
We endow any Borel set $\Z\subseteq\R^d$ with the standard topology, and we denote as $\PR_\Z$ the set of Borel probability measures on $\Z$. For $z\in\Z$, $\delta_z$ is the Dirac unit mass on $z$. For $P\in\PR_\Z$ and a Borel measurable function $f$ on $\Z$, $\E_P[f(Z)]$ (or more compactly $\E_P[f]$) denotes the expectation of $f$ under $Z\sim P$. Given a sub-sigma-field $\mathcal G$, $\E_P[Z|\mathcal G]$ is the conditional expectation. We will be interested in the case of $\mathcal G$ being the sigma-field generated by some random variable $X$. In such case, we write $\E_P[Z|X]$. 

We  also consider measures on product spaces. For $T\geq 1$, we write $\PR_{\Z^T}$ for the set of Borel probability measures on $\Z^T$ (endowed with the product topology). We will use $\otimes$ to denote the direct product of measures. For instance, given $P$ and $Q$ in $\PR_\Z$, $P\otimes Q$ will be the element of $\PR_{\Z^2}$ that encodes the law of $(Z_1,Z_2)$, where $Z_1\sim P$ and $Z_2\sim Q$ are independent. 

For any two given integers $s$ and $t$, with $s\leq t$, $[s:t]$ denotes the set of integers between $s$ and $t$ (both included). For an integer $T$, given a vector $(z_1, \dots, z_T)$, we often represent it compactly as $z^T$ (upper indices). At times, we will also use the notation $z^{t:T}$ (with $t\in[1:T]$), to denote the vector $(z_t, z_{t+1}, \dots, z_T)$. Sequences are denoted as $(s_t)_{t\geq T_0}$, with $t$ an integer index and $T_0$ its smallest value (typically $0$ or $1$). Sometimes we will also use the notation $s^{T_0:\infty}$ to denote $(s_t)_{t\geq T_0}$, or simply $s^\infty$ if $T_0$ is clear from the context. For high probability statements, $\Pp$ expresses probability with respect to all the randomness involved. For instance, if $(Z_t)_{t\geq 1}$ is a sequence of i.i.d.\ draws from  $P\in\PR_\Z$, we may write $\Pp(Z_t\geq1/2\,,\,\forall t\geq 1)$, with obvious meaning.

\section{Algorithmic mean testing via single-round e-variables}\label{sec:ame}
We start by presenting a framework for sequential hypothesis testing, formalised as a betting game. We then specialise to testing the mean of a bounded distribution, introducing the coin-betting approach.

\subsection{Sequential testing game}
Let $\Z$ be a non-empty Borel set in $\R^d$. A \emph{hypothesis} on $\Z$ is a non-empty subset $\Hh$ of $\PR_\Z$, and an \emph{e-variable} (for $\Hh$) is a non-negative Borel measurable function $E:\Z\to[0,+\infty)$, such that
$$\E_P[ E] \leq 1\,,\qquad\forall P\in\Hh\,.$$
We denote as $\Ee_\Hh$ the set of all the e-variables with respect to $\Hh$, and we call \emph{e-class} any subset of $\Ee_\Hh$. We remark that $\Ee_\Hh$ is never empty, since the constant function $1$ is always an e-variable, for any $\Hh$.

\begin{definition}[Testing-by-betting game]\label{def:game}
    Fix a hypothesis $\Hh\subseteq\PR_{\Z}$ and a non-empty e-class $\Ee\subseteq\Ee_\Hh$. An \emph{$\Ee$-restricted testing-by-betting game} (on $\Hh$) is the following sequential procedure. Each round $t\geq 1$, 
    \begin{itemize}\setlength{\itemsep}{0pt}
        \item the player picks\footnote{We assume that $E_t$ is picked in a measurable fashion, namely the mapping $(z_1,\dots z_t)\mapsto E_t(z_t)$ is Borel measurable.} an e-variable $E_t\in\Ee$ based solely on the past observations $z_1,\dots, z_{t-1}$;
        \item the player observes a new data-point $z_t\in\Z$;
        \item the player earns a reward $\log E_t(z_t)$. 
    \end{itemize}
    If $\Ee=\Ee_\Hh$, we speak of \emph{unrestricted} testing-by-betting game.
\end{definition}
The above game is an instance of e-variable testing, where one designs a test that rejects the hypothesis $\Hh$ whenever the total reward earned by the player gets excessively high. This procedure is justified by the fact that, if the data-points observed during the game were independently drawn from  $P \in \Hh$, then the cumulative reward would be unlikely to grow very large. This is formalised by the following proposition.
\begin{proposition}\label{lemma:game}
    Let $\Hh\subseteq\PR_{\Z}$ and consider a sequence $(Z_t)_{t\geq 1}$ of independent draws from  $P\in\Hh$. Fix $\delta\in(0,1)$ and $\Ee \subseteq \Ee_\Hh$. Consider an $\Ee$-restricted testing-by-betting game, where the observations are the sequence $(Z_t)_{t\geq 1}$. Let $R_n=\sum_{t=1}^n\log E_t(Z_t)$ represent the player's cumulative reward at round $n$. Then,
    $$\Pp\big(R_n\leq\log\tfrac{1}{\delta}\,,\;\forall n\geq 1\big)\geq 1-\delta\,.$$
\end{proposition}
\begin{proof}
    The result follows directly from Ville's inequality, since $M_n = \prod_{t=1}^n E_t(Z_t)$ defines a non-negative super-martingale with respect to the natural filtration of the process $(Z_t)_{t\geq 1}$, with $M_0\equiv 1$.\qed
\end{proof}
The cumulative reward earned by the player serves as a quantitative measure of evidence against the hypothesis $\Hh$. Given a sequence of independent observations known to be drawn from some $P\in\PR(\Z)$, \Cref{lemma:game} justifies the following sequential testing procedure: the null hypothesis ``\emph{the data generating distribution $P$ is in $\mathcal{H}$}'' is rejected as soon as the player's total reward exceeds the threshold $\log(1/\delta)$, for a chosen confidence level $\delta \in (0,1)$. In this setup, $\delta$ controls the Type I error rate, ensuring that the probability of wrongly rejecting a true null is at most $\delta$. Remarkably,  \Cref{lemma:game} guarantees this control uniformly over time, allowing the statistician to freely decide when to stop the test. For more details on sequential testing by betting, we refer to \cite{ramdas2023game}, or Chapter 7 of \cite{ramdas2024hypothesis}.

To design a powerful test, we want the rewards to accumulate rapidly whenever the data provide evidence against the null. To this regard, the pool $\Ee$, from which the player can pick the e-variables, plays an important role: excluding useful functions may weaken the test, while including unnecessary ones (e.g., the constant $1/2$) adds no value. A carefully tailored class can simplify strategy design while preserving statistical power. The goal of this paper is to identify the ``best'' e-class to use when testing for the mean of a bounded real random variable.

\subsection{Testing for the mean and coin-betting}
Fix a Borel set $\X\subseteq[0,1]$, containing $0$ and $1$.\footnote{The main results of this work are actually valid for any Borel set whereof $[0,1]$ is the convex closure. The requirement that $0$ and $1$ belong to $\X$ slightly simplifies some proofs (e.g., \Cref{lemma:F}), which use the fact that $0$ and $1$ are in $\X$.}
Let $(X_t)_{t\geq 1}$ be a sequence of independent random variables drawn from an unknown fixed distribution $P^\star \in \PR_\X$, with mean $\mu^\star\in(0,1)$. To test whether $\mu^\star$ equals a given value $\mu \in (0,1)$, we can define the corresponding null hypothesis:
\begin{equation}\label{eq:Hmu}
\Hh_\mu = \{P \in \PR_{\X} : \E_P[X] = \mu\}\,,
\end{equation}
which is the set of all distributions on $\X$ having mean $\mu$. Rejecting $\Hh_\mu$ thus amounts to rejecting the claim that $\mu^\star = \mu$. Such null hypotheses are closely related to the problem of mean estimation, specifically, to constructing a sequence of intervals that, with high probability, contain the true mean (i.e., a \emph{confidence sequence}). We will make this connection more precise in \Cref{sec:conf}. We stress here that, for the time being, we assume that the sequence $(X_t)_{t\geq 1}$ is i.i.d., and we shall not challenge this assumption, regardless of whether or not the observed data \emph{appear} to exhibit such behaviour. In a way, this perspective aligns naturally with the goal of constructing confidence sequences that we will discuss in \Cref{sec:conf}, as simultaneously testing for the mean and the i.i.d.~assumption might lead to rejecting every point in $\X$ and returning empty confidence sets because ``\emph{the data do not look i.i.d.~enough}''. As a matter of facts, under the i.i.d.~model, mean estimation is well posed: we assume the existence of a fixed distribution $P^\star$ independently generating each observation, and our task is to estimate its mean. We remark that the independence assumption could be relaxed by just asking that each $X_t$ has fixed conditional mean $\mu^\star$ (to be estimated) given the past $\F_{t-1}$. The arguments we present next for the i.i.d.~setting carry over with essentially no change to this less restrictive case. However, to keep the exposition clearer, we focus here solely on the independent case.  The conditional setting will be addressed in the second part of this work (\Cref{sec:multi}), where we consider the more complex scenario in which the nature of the depedences between observations is  not given or assumed, but is instead challenged via statistical testing.

\subsubsection{The coin-betting e-class}

Recently, \cite{orabona2023tight} and \cite{waudbysmith23estimating} obtained some of the tightest known confidence sequences for the mean of bounded real random variables via an algorithmic approach based on sequential testing. The main idea behind both papers involves the following sequential testing game for $\Hh_\mu$, which can be thought as betting on the outcome of a ``continuous'' coin (see \citealt{orabona2023tight} for a thorough discussion on the ``coin-betting'' interpretation). 
\begin{definition}[Coin-betting game]\label{def:coin}
    Fix $\mu\in(0,1)$ and let $I_\mu=[(\mu-1)^{-1}, \mu^{-1}]$.\footnote{The definition of $I_\mu$ ensures that the game's rewards are well defined, as logarithms of non-negative quantities.} Consider the following sequential procedure. At each round $t\geq 1$, a player 
    \begin{itemize}\setlength{\itemsep}{0pt}
        \item picks\footnote{Again, we implicitly assume a measurable selection of $\beta_t$, namely the mapping $(x_1,\dots x_{t-1})\mapsto \beta_t$ is Borel measurable.} $\beta_t\in I_\mu$, based solely on the past observations $x_1,\dots, x_{t-1}$;
        \item observes a new data-point $x_t\in[0,1]$;
        \item receives the reward $\log \big(1+\beta_t(x_t-\mu)\big)$. 
    \end{itemize}
\end{definition}
We remark that this coin-betting game is a specific instance of the testing-by-betting game that we have  described earlier. Indeed, letting $E_t : x \mapsto 1 + \beta_t(x-\mu)$, it is straightforward to verify that $E_t$ is non-negative on $\X$, due to the restriction $\beta_t \in I_\mu=[(\mu-1)^{-1}, \mu^{-1}]$ in \Cref{def:coin}. Moreover, for any $P \in \mathcal{H}_\mu$, we have $\E_P[ E_t ] = 1$, which implies that $E_t$ is an e-variable for $\Hh_\mu$. Finally, the reward in the coin-betting game is precisely equal to $\log E_t(x_t)$. Hence, for the hypothesis $\Hh_\mu$, the coin-betting game above matches exactly the testing game of \Cref{def:game}, restricted to the \textit{coin-betting e-class}
\begin{equation}\label{eq:Ecb}\Ecb_\mu = \big\{E_\beta:x\mapsto 1+\beta(x-\mu)\,,\;\;\beta\in I_\mu\big\}\,.\end{equation}

\subsubsection{A suboptimal choice: the Hoeffding e-class}
Another perhaps natural e-class to test $\Hh_\mu$ is  
$$\EH_\mu = \{\EHa:x\mapsto e^{\alpha(x-\mu)-\alpha^2/8}\,,\;\;\alpha\in\R\}\,,$$
which we will refer to as the \emph{Hoeffding e-class}. The fact that this is an e-class follows immediately from the well known sub-Gaussian nature of bounded random variables (Hoeffding's lemma). 

Let us now consider two players, Alice and Bob, playing two different testing-by-betting games for $\Hh_\mu$. Alice plays a $\Ecb_\mu$-restricted game, while Bob a $\EH_\mu$-restricted game. We will now show that Alice's game is strictly stronger, from a statistical perspective, than Bob's one. More precisely, if Bob plays first, and picks $\EHa$, there is always a $\lambda_\alpha\in I_\mu$ such that, picking $E_{\lambda_\alpha}$,  Alice can be sure of getting a reward that is at least as high as Bob's one, no matter what the observation that round will be. Notably, the converse is not true: there are  $\lambda\in I_\mu$ such that no e-variable $\EHa\in\EH_\mu$  dominates $E_\lambda$ for every possible value of $x$. This is formalised in the next statement, whose simple proof relies on elementary calculus (see \ref{app:hoef}).
\begin{proposition}\label{prop:hoef}
    Fix $\mu\in(0,1)$. For each $\alpha\in\R$ there is $\lambda_\alpha\in I_\mu$ such that $E_{\lambda_\alpha}(x)\geq \EHa(x)$ for all $x\in\X$. On the other hand, fix any non-zero $\lambda\in I_\mu$. For every $\alpha\in\R$, there is at least a point $x_\alpha\in\X$ such that $\EHa(x_\alpha)<E_\lambda(x_\alpha)$.
\end{proposition}
As a consequence of the above discussion, when testing $\Hh_\mu$ it is always ``better'' to restrict the player to $\Ecb_\mu$ rather than to $\EH_\mu$. We will make this intuition more formal in the next section.

\section{Majorising e-classes and optimal e-class}
Ideally, one aims to set up a powerful testing procedure, capable of rejecting the null as soon as there is enough evidence against it. However, achieving this depends on the strategy employed in the testing games. For example, stubbornly playing $E_t \equiv 1$ at all rounds in every game would produce powerless tests of no practical interest. Indeed, a player's strategy is most effective when it can rapidly increase the cumulative reward, whenever possible. In short, \emph{the highest the rewards, the more powerful the statistical test}. We have already seen at the end of the previous section that there is no point in considering a $\EH_\mu$-restricted game, as this is always outperformed by the $\Ecb_\mu$-restricted one. An even worse option would be the trivial restriction to $\Ee=\{1\}$, which can never lead to rejection. With this in mind, it is clear that carelessly restricting the player's choice to a subset  of $\Ee_{\Hh_\mu}$ could be highly detrimental, as it might force the player to adopt poor strategies. This point naturally raises the question: \emph{``Does restricting the player to the coin-betting e-class \eqref{eq:Ecb} loosen the confidence sequence?''}. Interestingly, for the round-wise testing-by-betting framework that we are considering the answer turns out to be negative. In order to make this statement rigorous we now introduce the concepts of majorising and optimal e-class.

First, we endow the set of real functions on a set $\Z$ with a partial ordering. Given two functions $f, f':\Z\to\R$, we say that $f$ \emph{majorises} $f'$, and write $f\succeq f'$, if $f(z)\geq f'(z)$ for all $z\in\Z$. If $f\succeq f'$ and there is a $z\in\Z$ such that $f(z)>f'(z)$, we say that $f$ is a \emph{strict} majoriser of $f'$, and we write $f\succ f'$. 

\begin{definition}An e-variable $E\in\Ee_\Hh$ is called   \emph{maximal} if there is no $E'\in\Ee_\Hh$ such that $E'\succ E$.
\end{definition}

Next, we introduce a way to compare different e-classes.
\begin{definition}
    Given two e-classes $\Ee$ and $\Ee'$, we say that $\Ee$ \emph{majorises} $\Ee'$ if, for any $E'\in\Ee'$, there is an e-variable $E\in\Ee$ such that $E\succeq E'$. An e-class is said to be a \emph{majorising} e-class if it majorises $\Ee_\Hh$. 
\end{definition}
\begin{lemma}\label{lemma:maj}
    Every majorising e-class contains all the maximal e-variables.
\end{lemma}
\begin{proof}
    Let $E\in\Ee_\Hh$ be maximal and $\Ee$ a majorising e-class. There must be $E'\in\Ee$ such that $E'\succeq E$, but since $E$ is maximal it has to be that $E=E'$. Hence, $E\in\Ee$.\qed
\end{proof}
The significance of the notion of majorising e-class for our problem is straightforward: if $\Ee$ majorises $\Ee'$, then any strategy in an $\Ee'$-restricted game can be matched or outperformed (in terms of rewards) by a corresponding strategy in the $\Ee$-restricted game, regardless of the sequence of observations. This allows us to compare how the restriction to different e-classes affects the testing-by-betting game of \Cref{def:game}. Notably, \Cref{prop:hoef} implies that the coin-betting e-class always majorises the Hoeffding e-class.

\begin{definition}
    If a majorising e-class is contained in every other majorising e-class, it is called \emph{optimal}.
\end{definition}
For any $\Hh$, a majorising e-class always exists, as $\Ee_\Hh$ itself is a majorising e-class. However, an optimal e-class may not exist.\footnote{We refer  to the follow-up work \cite{clerico24optimal} for an example of non-existence of the optimal e-class. Proposition 2 therein shows that $\Hh = \{P\in\PR_{\X}\,:\,P(\{0\})\geq 1/2\}\cup\{U_{[0,1]}\}$ (with $U_{[0,1]}$ the uniform distribution on $[0,1]$) is a hypothesis for which the set of maximal e-variables is not a majorising e-class. Hence, $\Hh$ does not admit an optimal e-class.} Next, we state a sufficient and necessary condition for its existence. 

\begin{lemma}\label{lemma:maxopt}
    An optimal e-class exists if, and only if, the set of all maximal e-variables is a majorising e-class. If an optimal e-class exists, it is unique, it corresponds to the set of all maximal e-variables, and it is the only majorising e-class whose elements are all maximal. 
\end{lemma}
\begin{proof}
    Denote as $\hat\Ee$ the set of all the maximal e-variables. Assume that there exists an optimal e-class $\Ee$. Let us show that all its elements are maximal. For  $E\in\Ee$, consider any element $E'\in\Ee_\Hh$ such that $ E'\succeq E$. We can construct an e-class $\Ee'$  replacing $E$ with $E'$ in $\Ee$, namely  $\Ee' = (\Ee\setminus\{E\})\cup\{E'\}$. Since $E'\succeq E$, it is clear that $\Ee'$ majorises $\Ee$. So, $\Ee'$ is a majorising e-class, as $\Ee$ is. Since $\Ee$ is optimal, $\Ee\subseteq\Ee'$, which implies $E=E'$. In particular, $E$ does not have any strict majoriser, and so it is maximal. In particular, $\Ee\subseteq\hat\Ee$. As $\Ee\supseteq\hat\Ee$ by \Cref{lemma:maj}, we conclude that $\Ee=\hat\Ee$, and so $\hat\Ee$ is a majorising e-class.

    Conversely, assume that $\hat\Ee$ is a majorising e-class. Let $\Ee$ be any other majorising e-class. By \Cref{lemma:maj}, $\hat\Ee\subseteq\Ee$. So, $\hat\Ee$ is contained in all the majorising e-classes, and hence it is optimal.
    
    Now, the remaining statements are a trivial consequence of what was shown above and \Cref{lemma:maj}.\qed
\end{proof}

Let us emphasise once more that, from our discussion thus far, it is clear that restricting the testing-by-betting game of \Cref{def:game} to a majorising e-class does not hinder the performance of the player, as for any unrestricted strategy $(E_t)_{t\geq 1}$ they can always pick a restricted strategy $(E_t')_{t\geq 1}$, whose cumulative rewards inevitably match or outperform those of $(E_t)_{t\geq 1}$, regardless of the sequence of observations. From a practical perspective, identifying the optimal e-class, when it exists, greatly simplifies the design of an effective strategy by narrowing the player's choice to the best possible e-variables. Specifically, if the optimal e-class exists and a player chooses an e-variable $E_t$ outside of it, they could always have picked an alternative $E_t'\succ E_t$, within the optimal e-class, whose reward is never worse than that of $E_t$ and is strictly higher for at least one possible value that $x_t$ might take. Conversely, when a player selects an e-variable from the optimal e-class, no other choice can be guaranteed to be better before observing $x_t$, since the player's pick is a maximal e-variable. As a straightforward consequence, within this round-wise testing-by-betting approach, the optimal approach to test $\Hh_\mu$ consists in restricting game to the optimal e-class. We show next that this coincides precisely with the coin-betting formulation.

\section{Optimality of the coin-betting e-class}\label{sec:optcoin}
For any $\mu\in(0,1)$, define the hypothesis $\Hh_\mu$ as in \eqref{eq:Hmu}. We now show that the optimal e-class for $\Hh_\mu$ exists and coincides with the coin-betting e-class $\Ecb_\mu$, defined in \eqref{eq:Ecb}. First, let us show that each e-variable is majorised by the function $F_\mu = \max(E_{\mu^{-1}}, E_{(\mu-1)^{-1}})$. 

\begin{lemma}\label{lemma:F}
    Fix $\mu\in(0,1)$ and consider the function $F_\mu:\X\to[1,+\infty)$ defined as 
    \begin{figure}[b!]
        \centering
        \begin{tikzpicture}
\def\m{0.58}
\def\bzero{-1.3955}
\def\buno{1.045}
\def\xzero{0.25}
\def\xuno{0.9}

  \begin{axis}[
    domain=0:1,
    samples=100,
    axis lines=middle,
    axis on top,
    xlabel={$x$},
    xtick={\xzero,\m,\xuno,1},
    xticklabels={$u_0$, $\mu$, $u_1$, },
    ytick={1},
    ymin=-0.25, ymax=2.75,
    xmin=-.1, xmax=1.25,
    width=9cm,
    height=7cm,
    thick,
    enlargelimits=false,
    clip=false,
    axis line style={->},
    tick style={thick},
  ]

    \addplot [
      fill=salmon!30, draw=none
    ] coordinates {
      (\m,1)
      (1,{1/\m})
      (1,{1 + \buno*(1 - \m)})
    };
    \addplot [
      fill=salmon!30, draw=none
    ] coordinates {
      (\m,1)
      (0,0)
      (0,{1 + \buno*(0 - \m)})
    };

    \addplot [
      fill=apricot!20, draw=none
    ] coordinates {
      (\m,1)
      (0,{-1/(\m-1)})
      (0,{1 + \bzero*(0 - \m)})
    };
    \addplot [
      fill=apricot!20, draw=none
    ] coordinates {
      (\m,1)
      (1,0)
      (1,{1 + \bzero*(1 - \m)})
    };

    \addplot[salmon, dashed, domain=0:1, very thick] {1+\buno*(x-\m)};
    \addplot[apricot, dashed, domain=0:1, very thick] {1+\bzero*(x-\m)};

    \addplot[purple, thick, domain=0:1] {1-(x-\m)/2};
    \node[purple, anchor=south west, font=\scriptsize] at (axis cs:1,.65) {$1\!+\!\beta^\star\!(x\!-\!\mu)$};
    \node[salmon, anchor=south west, font=\scriptsize] at (axis cs:1,1.3) {$1\!+\!\beta_1\!(x\!-\!\mu)$};
    \node[apricot, anchor=south west, font=\scriptsize] at (axis cs:1,.3) {$1\!+\!\beta_0\!(x\!-\!\mu)$};
    \node[black, anchor=center, font=\scriptsize] at (axis cs:.2,2.1) {$F_\mu$};
    \node[olive, anchor=center, font=\scriptsize] at (axis cs:.48,.57) {$E$};
    \node[pastelblue, anchor=center, font=\scriptsize] at (axis cs:.2,1.3) {$U_0$};
    \node[pastelblue, anchor=center, font=\scriptsize] at (axis cs:.96,1.2) {$U_1$};
    \node[red, anchor=center, font=\scriptsize] at (axis cs:\m,1.4) {$U^\star$};

    \node[black, anchor=center] at (axis cs:1,-0.14) {$1$};
    \node[black, anchor=center] at (axis cs:-.05,-0.14) {$0$};

    \addplot[domain=\m:1, black, very thick] {x/\m};

    \addplot[domain=0:\m, black, very thick] {(x - 1)/(\m - 1)};

    \pgfmathdeclarefunction{E}{1}{%
         \pgfmathparse{(-305/7)*#1^4 + (1811/21)*#1^3 - (7379/140)*#1^2 + (793/84)*#1 + (11/10)}%
    }

    \pgfmathsetmacro{\yzero}{E(\xzero)}
    \pgfmathsetmacro{\yuno}{E(\xuno)}
    \pgfmathsetmacro{\ym}{\yzero + (\yuno - \yzero)*(\m - \xzero)/(\xuno - \xzero)}
    
    \addplot[olive, very thick, domain=0:1] {E(x)};

    \addplot[only marks, mark=*, mark size=1.5pt, pastelblue] coordinates {
        (\xzero, \yzero)
        (\xuno, \yuno)
    };

    \draw[dashed, thin, pastelblue] (axis cs:\xzero,\yzero) -- (axis cs:\xuno,\yuno);
    \draw[dashed, thin] (axis cs:0,1) -- (axis cs:1,1);
    \draw[dotted, thin] (axis cs:\xzero,\yzero) -- (axis cs:\xzero,0);
    \draw[dotted, thin] (axis cs:\xuno,\yuno) -- (axis cs:\xuno,0);
    \draw[dotted, thin] (axis cs:\m,\ym) -- (axis cs:\m,0);

    \addplot[only marks, mark=x, mark size=3pt, red] coordinates {
        (\m, \ym)
    };
    \end{axis}

\end{tikzpicture}
        \caption{Pictorial representation of the main step in the proof of \Cref{thm:main}. $1+\beta_1(x-\mu)$ dominates $E$ for $x\in[0,\mu)$, while $1+\beta_0(x-\mu)$ dominates $E$ for $x\in(\mu, 1]$. If there is $\beta^\star\in(\beta_0,\beta_1)$, then we can find $u_0\in[0,\mu)$ and $u_1\in(\mu,1]$ such that $E(u_0)>1+\beta^\star(u_0-\mu)$ and $E(u_1)>1+\beta^\star(u_1-\mu)$, represented by the points $U_0$ and $U_1$ being above the purple line $1+\beta^\star(x-\mu)$. The probability measure $\hat P$ supported on $\{u_0,u_1\}$ with mean $\mu$ is in $\Hh_\mu$. We have $\E_{\hat P}[E]>1$. Indeed, this expected value corresponds to the vertical coordinate of the point $U^\star$, the intersection of the line connecting $U_0$ and $U_1$ with the vertical line at $x=\mu$. This is a contradiction if $E$ is an e-variable, in which case it must be that $\beta_0\geq\beta_1$.}
        \label{fig:mu}
    \end{figure}
    $$F_\mu\;:\;x\mapsto\begin{cases} 1 + \frac{1}{\mu}(x-\mu)&\text{if $x\geq\mu$;}\\ 1 + \frac{1}{\mu-1}(x-\mu)&\text{if $x<\mu$.}\end{cases}$$ For any $x\in\X$ there is $P_x\in\Hh_\mu$ such that $P_x(\{x\})= 1/F_\mu(x)>0$. Moreover, $F_\mu\succeq E$ for all $E\in\Ee_{\Hh_\mu}$.
\end{lemma}
\begin{proof}
    For $x\in\X\cap[\mu,1]$, let $P_x = \frac{\mu}{x}\delta_x + (1-\frac{\mu}{x})\delta_0$. Then, $P_x\in\Hh_\mu$ and $P_x(\{x\}) = \mu/x = 1/F_\mu(x)$. Similarly, if $x<\mu$ we can find a measure $P_x$ in $\Hh_\mu$, supported on $\{x,1\}$, with mass $1/F_\mu(x)$ on $x$. To check that $F_\mu$ majorises all the e-variables, fix $E\in\Ee_{\Hh_\mu}$ and $x\in\X$. Let  $P_x\in\Hh_\mu$ have mass $1/F_\mu(x)$ on $x$. Then, $1\geq \E_{P_x}[E] \geq P_x(\{x\})E(x) = E(x)/F_\mu(x)$, and we conclude. \qed
\end{proof}

\begin{theorem}\label{thm:main}
    For any $\mu\in(0,1)$, the coin-betting e-class $\Ecb_\mu$ is the optimal e-class for $\Hh_\mu$.
\end{theorem}
\begin{proof}
    First, let us show that $\Ecb_\mu$ is a majorising e-class. For a pictorial representation of this part of the proof, the reader is invited to look at \Cref{fig:mu}. Fix an arbitrary $E\in\Ee_{\Hh_\mu}$. Define the sets
    $$B_0 = \left\{\beta\in I_\mu\;:\; \inf_{x\in\X\cap[0,\mu)}\big(E_\beta(x)-E(x)\big)\geq0\right\}\;\text{and}\,\; B_1 = \left\{\beta\in I_\mu\;:\; \inf_{x\in\X\cap(\mu,1]}\big(E_\beta(x)-E(x)\big)\geq0\right\}\,,$$ where we recall that $I_\mu = [(\mu-1)^{-1}, \mu^{-1}]$ and $E_\beta:x\mapsto 1+\beta(x-\mu)$. Both sets are closed and convex (as they are intersections of closed and convex sets). By \Cref{lemma:F}, $(\mu-1)^{-1}\in B_0$ and $\mu^{-1}\in B_1$, so $B_0 = [(\mu-1)^{-1},\beta_0]$ and $B_1 = [\beta_1, \mu^{-1}]$, for some $\beta_0$ and $\beta_1$ in $I_\mu$. We will now show that $B_0\cap B_1\neq\varnothing$, or equivalently that $\beta_0\geq\beta_1$. Assume that this was not the case and $\beta_0<\beta_1$. Let $\beta^\star\in(\beta_0, \beta_1)$. Then, $\beta^\star\notin B_0$ and $\beta^\star\notin B_1$. In particular, there are $u_0<\mu$ and $u_1<\mu$, in $\X$, such that $E(u_0)>E_{\beta^\star}(u_0)$ and $E(u_1)>E_{\beta^\star}(u_1)$.  As $\mu\in(u_0, u_1)$, there is $\hat P\in\Hh_\mu$ with support $\{u_0,u_1\}$. Note that $\E_{\hat P}[E_{\beta^\star}] = 1 + \beta^\star(\E_{\hat P}[X] -\mu) = 1$. But $E$ is strictly larger than $E_{\beta^\star}$ on $\supp(\hat P)$, and so $\E_{\hat P}[E] >1$, which is a contradiction since $E$ is an e-variable. So, $\beta_0\geq\beta_1$, and there exists $\hat\beta\in B_0\cap B_1$. By construction, $E_{\hat\beta}\in\Ecb_\mu$ and $E_{\hat\beta}(x)\geq E(x)$ for all $x\in\X$ different from $\mu$. If  $x=\mu$ by \Cref{lemma:F} we have $E(\mu)\leq F_\mu(\mu) = 1 = E_{\hat\beta}(\mu)$, so $E\preceq E_{\hat\beta}$. As the choice of $E$ was arbitrary, $\Ecb_\mu$ is a majorising e-class. 
    
    Once established that $\Ecb_\mu$ is a majorising e-class, by \Cref{lemma:maxopt} we only need to show that all its elements are maximal. Fix $E\in\Ecb_\mu$, and consider an e-variable $\hat E\in\Ee_{\Hh_\mu}$ such that $\hat E\succeq E$. Fix any $x\in\X$. By \Cref{lemma:F}, there is $P\in\Hh_\mu$ such that $P(\{x\})>0$. Since $E\in\Ecb_\mu$, we have $\langle P, E\rangle =1$, and so $0\leq P(\{x\})(\hat E(x)-E(x))\leq\E_P[\hat E -E] =\E_P[\hat E] -1\leq 0$. Since $P(\{x\})>0$, we get $\hat E(x)=E(x)$ and, $x$ being arbitrary,  $\hat E=E$. Hence, $E$ is maximal, as it has no strict majoriser. \qed
\end{proof}

\section{From mean testing to confidence sequences}\label{sec:conf}
Before moving to the more complex setting of multi-round e-variables and e-processes, which allow for tests that challenge the dependence structure of the observations, we first illustrate how sequential testing can be directly applied to the problem of mean estimation. In particular, we show how this framework naturally gives rise to confidence sequences, a sequential counterpart to classical confidence intervals that dates at least back to \cite{darling67confidence}. We consider the following approach to constructing confidence sequences via hypothesis testing: at each time step, we  test  each candidate value $\mu$ for the mean, and include in the confidence set those values that are not rejected. For further discussion on the connection between sequential testing and confidence sequences, we refer  to \cite{ramdas2022admissible}. 

As usual, $(X_t)_{t\geq 1}$ is a sequence of independent draws from $P^\star\in\PR(\X)$, whose mean $\mu^\star\in(0,1)$ has to be estimated. Let $\mathcal{F} = (\F_t)_{t\geq 0}$ represent the natural filtration generated by $(X_t)_{t\geq 1}$, where $\F_t = \sigma(X_1, \dots, X_t)$ captures all information available up to time $t$. Fix a confidence level parameter $\delta \in (0,1)$. A \emph{confidence sequence} $(S_t)_{t\geq 1}$ is a sequence of random sets\footnote{Here, by random sets we simply mean sets that depend on the sequence of random observations $(X_t)_{t\geq 1}$.} such that the sequence of events $(\{\mu^\star\in S_t\})_{t\geq 1}$ is adapted  to the filtration $\F$ (i.e., for all $t\geq 1$, $\{\mu^\star\in S_t\}$ is $\F_{t}$-measurable) and satisfies  $$\Pp\big(\mu^\star\in S_t\,,\;\forall t\geq 1\big) \geq 1-\delta\,.$$
Intuitively, this means that $(S_t)_{t\geq 1}$ provides a set of plausible values for $\mu^\star$ at each time step $t$, while ensuring that the true mean remains in these sets indefinitely with high probability.

We can leverage the testing-by-betting game of \Cref{def:game} to obtain a confidence sequence for the mean $\mu^\star$ of $P^\star$. For  $\mu\in(0,1)$, we define the hypothesis $\Hh_\mu$ as in \eqref{eq:Hmu}, namely $\Hh_\mu$ contains all probability measures on $\X$ with mean $\mu$. For each $\mu$, we fix an e-class $\Ee_\mu\subseteq\Ee_{\Hh_\mu}$ and we consider an $\Ee_\mu$-restricted testing-by-betting game on $\Hh_\mu$, where the player observes $(X_t)_{t\geq 1}$, the sequence of draws from $P^\star$. For each one of these games, denote as $R_n(\mu)$ the player's cumulative reward at round $n$. We can then construct a confidence sequence as follows.
\begin{proposition}
    The sequence $(S_n)_{n\geq 1}$, defined as $$S_n = \big\{\mu\in(0,1)\;:\;R_n(\mu)\leq\log\tfrac{1}{\delta}\big\}\,,$$ is a confidence sequence for the mean $\mu^\star$ of the data-generating probability measure $P^\star$.
\end{proposition}
\begin{proof}
    For each $n\geq 1$, we have that $\{\mu^\star\in S_n\}=\{R_n(\mu^\star)\leq\log\frac{1}{\delta}\}$, which is a $\F_{n}$-measurable event. Moreover, $$\Pp\big(\mu^\star\in S_n\,,\;\forall n\geq1 \big) = \Pp\big(R_n(\mu^\star)\leq\log\tfrac{1}{\delta}\,,\;\forall n\geq1 \big)\geq1-\delta$$ by \Cref{lemma:game}, as the observations are drawn from $P^\star\in\Hh_{\mu^\star}$. \qed
\end{proof}
It is worth stressing that the strength of the resulting confidence sequence depends directly on the power of the underlying sequential tests for each hypothesis $\Hh_\mu$: more powerful tests yield tighter confidence sets. In particular, the optimality result for the coin-betting e-class established in \Cref{sec:optcoin} implies that, for each $\mu$, the testing game should be restricted to the class $\Ecb_\mu$. With this choice, we recover the coin-betting framework to derive confidence sequences leveraged by \cite{orabona2023tight} and \cite{waudbysmith23estimating}. 

Both \citet{orabona2023tight} and \citet{waudbysmith23estimating} propose explicit strategies for placing bets in the coin-betting game, leading to concrete confidence sequences. While the present work's focus is on defining the optimal betting game rather than designing specific strategies, it is still useful to look more closely at the approach of \citet{orabona2023tight} to  illustrate how the optimality of the coin-betting e-class simplifies the construction of a strategy  for the game of \Cref{def:game}. They employ a coin-betting strategy based on the universal portfolio algorithm, a special instance of online learning Bayesian aggregation techniques (see, e.g., Chapters 9 and 10 of \citealp{cesabianchi06prediction}). In this approach, a prior distribution $\rho_1$ is fixed over the decision space $I_\mu$, and then updated sequentially using observed data. Specifically, at round $t$, one defines $$\dd\rho_t(\lambda) = \zeta_t^{-1}\prod_{i=1}^{t-1}(1+\lambda x_i)\dd\rho_1(\lambda)\,,$$
where $\zeta_t$ is the normalising constant ensuring that $\rho_t$ is a  probability measure on $I_\mu$. The bet $\lambda_t$ is then chosen as the posterior mean: $\lambda_t = \int_{I_\mu} \lambda \dd\rho_t(\lambda)$.

This strategy can be interpreted as performing Bayesian averaging over the coin-betting e-class $\Ecb_\mu$, leveraging its simple and low-dimensional parametric structure. This is  a concrete example of how knowing explicitly the optimal e-class can help designing powerful testing-by-betting strategies. Indeed, such approach would not be feasible on the full class of all e-variables, an infinite dimensional functional space where even the definition of a prior can become problematic, if we did not know that the prior should be supported on the coin-betting e-class. Importantly, the optimality of $\Ecb_\mu$ ensures that no statistical power is sacrificed by restricting to this class. Notably, a similarly structured averaging strategy over a suboptimal e-class, such as the Hoeffding e-class or its convex hull, would remain well defined and computationally tractable, but lead to strictly worse performance, as shown by \Cref{prop:hoef}.  

\section{Optimality beyond single-round e-variables}\label{sec:multi}
Up to this point, we have focused on hypotheses defined as subsets of the space $\PR_{\X}$, which cannot capture depedences across multiple rounds. In such setting, the dependence structure among observations was fixed and assumed a priori, rather than being subject to testing. Yet, e-variable-based testing naturally extends to more general hypotheses that span multiple rounds and can account for sequential or dependent data structures. In the second part of the paper, we turn our attention to this richer framework. Remarkably, an  adaptation of the proof strategy used for \Cref{thm:main} and depicted in \Cref{fig:mu} allows us to show that coin-betting-based testing is also optimal in the setting where the sequence of observations has a fixed conditional mean $\mu$. We show that this is the case with multi-round e-variables for a fixed time horizon, and then extend the result to  testing with e-processes. Conversely, we will show that when testing the i.i.d.\ assumption, the coin-betting approach no longer yields the optimal e-class.

\subsection{Optimality with multi-round e-variables}
For any $\mu\in(0,1)$ and $T\geq 1$, we let 
$$\Hh_\mu^T = \big\{P\in\PR_{\X^T}\,:\,\E_P[X_t|X^{t-1}]=\mu\,,\;\forall t\in[1:T]\big\}\,,\footnote{The equality in $\E_P[X_t|X^{t-1}]=\mu$ has to be interpreted as holding $P$-almost everywhere.}$$ where $\E_P[X_1|X^0]=\E_P[X_1]$.
Note that $\Hh_\mu^T$ is a hypothesis on $\X^T$. As such, its set of e-variables will consist of Borel functions from $\X^T$ to $\R$. We denote the set of all e-variablues relative to $\Hh_\mu^T$ as $\Ee_\mu^T$. 

Let us consider a test where all the $T$ observations $x^T$ are seen at once, where a single e-variable $E\in\Ee_\mu^T$ needs to be selected. If $E(x^T)\geq 1/\delta$, the null $\Hh^T_\mu$ is rejected. We remark that this setting departs from the sequential testing-by-betting games discussed in the previous sections, as now the data set size $T$ is fixed in advance and the player makes a single decision before seeing the entire data set. In a way, this is as if we where playing a \emph{single} round in the game of \Cref{def:game}, with $\Z=\X^T$.\footnote{Of course, this can also be extended in a sequential game, where each round a new block  of $T$ data points is observed.}  Despite this difference, the connection to the $T$-round coin-betting game of \Cref{def:coin} remains strong. Indeed, in the coin-betting setup that we have considered earlier, the wealth of the player at round $T$ is a non-negative function of the observations, which takes the form $M_T(x^T) = \prod_{t=1}^n E_t(x_t)$, where $E_t$ is in the form $E_{\lambda_t}$, with $\lambda_t$ chosen as a function of the past observations $x^{t-1}$. It is straightforward to check that under the hypothesis $\Hh_\mu^T$, the expectation of $M_T$ is always exactly one ($M_T$ is a martingale), which shows that $M_T$ is an e-variable. Next, we show that the e-class of e-variables in this form coincides with the optimal e-class for $\Hh_\mu^T$. 

We let $\Lambda^T_\mu$ be a set of $T$-tuples of functions defined as
$$\Lambda^T_\mu = \big\{\lambda^T=(\lambda_1,\dots,\lambda_T)\,:\,\text{$\lambda_t$ is a Borel function from $\X^{t-1}$ to $I_\mu$}\big\}\,.\footnote{Here and henceforth, a function from $\X^0$ to $I_\mu$ is simply an element of $I_\mu$, so that whenever $\lambda_1(x^0)$ appears, it always has to be interpreted as an element $\lambda_1\in I_\mu$.}$$ Given $\lambda^T\in\Lambda^T_\mu$, we define the e-variable $E_{\lambda^T}\in\Ee^T_\mu$ as 
$$E_{\lambda^T}(x^T) = \prod_{t=1}^T\big(1+\lambda_t(x^{t-1})(x_t-\mu)\big)\,.$$
We define the \emph{T round coin-betting e-class} $\EcbT_\mu = \{E_{\lambda^T}\,:\,\lambda^T\in\Lambda_\mu^T\}$.

\begin{theorem}\label{thm:multi}
    Fix any $T\geq 1$ and $\mu\in(0,1)$. $\EcbT_\mu$ is the optimal e-class for $\Hh_\mu^T$. 
\end{theorem}
The proof of \Cref{thm:multi} shares many similarities with that of \Cref{thm:main}. In particular, we use the same technique to establish the maximality of e-variables of the form $E_{\lambda^T}$. A geometric argument analogous to the one illustrated in \Cref{fig:mu} can then be employed, as part of an induction recursion, to prove that the e-class of interest is majorising. However, the detailed proof is somewhat lengthy and involves some technical subtleties when dealing with measurability. We hence defer it to \ref{app:multi}.

\subsection{A remark on the i.i.d.~case} 
    From what we have established so far, coin-betting yields all and only the maximal e-variables when testing the $T$-round hypothesis $\Hh_\mu^T$ that the conditional mean is some fixed $\mu\in(0,1)$. Yet, this is not any more the case if we consider a more restrictive hypothesis, stating that the observations are i.i.d.~and with mean $\mu$. More concretely, let us define $$\hat\Hh_\mu^T = \big\{P=Q^{\otimes T}\,:\,Q\in\Hh_\mu\big\}\,$$
    and denote as $\hat\Ee_\mu^T$ the set of all the e-variables for $\hat\Hh_\mu^T$.
    Were $\EcbT_\mu$  to be the optimal e-class for $\hat\Hh_\mu^T$, then $\Ee_\mu^T$ and $\hat\Ee_\mu^T$ would have to coincide, as the optimal e-class completely determines the set of all the e-variables. However, this cannot be the case, unless $\X$ has only two elements. To avoid technicalities, let us consider the case where $\X$ has finite cardinality and has at least three elements. As $\X$ has finitely many elements, for any hypothesis $\Hh$, the largest (in an inclusion sense) hypothesis, whose e-variables are all and only the functions in $\Ee_\Hh$, is  precisely the closure of the closed convex hull of $\Hh$ \citep{larsson2024numeraire}. In particular, the optimality of $\EcbT_\mu$ for $\hat\Hh_\mu^T$ would imply that $\Hh_\mu^T$ is included in the convex hull of $\hat\Hh_\mu^T$, which is  not the case if $\X$ has at least three elements.\footnote{Note that if $\X=\{0,1\}$ has two elements only, then $\Hh_\mu^T$ and $\hat\Hh_\mu^T$ coincide.} It is however worth noticing that the e-variables in $\EcbT_\mu$ are still maximal for $\hat\Hh_\mu^T$. To see this, fix $E\in\EcbT_\mu$. We have that $E_P[E]=1$ for any $P\in\hat\Hh_\mu^T$. For every $x^T$, we can find a $P_{x^T}\in\hat\Hh_\mu^T$ such that $P_{x^T}(\{x^T\})>0$.\footnote{From \Cref{lemma:F} we know that for each $t\in[1:T]$ there is $Q_t\in\Hh_\mu$ such that $Q_t(\{x_t\})>0$. Let $Q=\frac{1}{T}\sum_{t=1}^TQ_t$. Then $P_{x^T} = Q^{\otimes T}\in\hat\Hh_\mu^T$ satisfies $P_{x^T}(\{x^T\})>0$.} By the same  argument that we used in the proof of \Cref{thm:main}, if an e-variable $\hat E$ majorises $E$, then for every $x^T$ we have $0\leq P_{x^T}(\{x^T\})(\hat E(x^T)-E(x^T))\leq \E_{P_{x^T}}[\hat E-E]=\E_{P_{x^T}}[\hat E]-1\leq 0$. So, $E=\hat E$, and $E$ is maximal.

    To give a concrete example of the mismatch between the e-variables for the conditional and independent scenarios, let us consider the ``highly symmetrical'' case  $\X = \{0,1/2,1\}$, with $\mu=1/2$ and $T=2$. We show in \ref{app:iid} that the e-variables in this setting are the  $E:\{0,1/2,0\}^2\to[0,+\infty)$  satisfying
    $$\xi_0\leq 1\,;\qquad\xi_1\leq 1+\sqrt{(1-\xi_0)(1-\xi_2)}\,;\qquad\xi_2\leq 1\,,$$ where $\xi_0 = E(1/2,1/2)$, $\xi_1 = (E(1,1/2)+E(1/2,1)+E(1/2,0)+E(0,1/2))/4$, and $\xi_2 = (E(1,1)+E(1,0)+E(0,1)+E(0,0))/4$. Clearly, the optimal e-class here consists of those e-variables that satisfy $\xi_0=\xi_1=\xi_2=1$, or $\xi_1= 1+\sqrt{(1-\xi_0)(1-\xi_2)}$ with $\xi_0<1$ and $\xi_2<1$. We note that all the e-variables in $\Ee_{1/2}^T$ satisfy $\xi_0=\xi_1=\xi_2=1$. As expected, these are indeed maximal e-variables for $\hat\Hh_{1/2}^2$. Yet, there are maximal e-variables for $\hat\Hh_{1/2}^2$ that are not in $\Ee_{1/2}^T$. The function that equals $4$ on $x^2= (1, 1/2)$, and $0$ everywhere else, is an example of a (maximal) e-variable for $\hat\Hh_{1/2}^2$, which is not an e-variable for $\Hh_{1/2}^2$.
    
    We leave as a (non-trivial) open question to characterise the e-variables for $\hat\Hh^T_\mu$ in the general case. 
\subsection{Optimal classes of e-processes}\label{sec:eproc}
A key advantage of many testing procedures involving e-variables is allowing early stopping. One way to achieve this is through the round-wise testing-by-betting approach that we presented in \Cref{sec:ame}, which gives rise to a super-martingale (cf.~\Cref{lemma:game}). However, as we have already pointed out, this is not the most general approach. Rather than only focusing on super-martingales, one can in principle use any non-negative process whose expectation is upper bounded by one, regardless the stopping time. These are known as e-processes, a generalization of super-martingales, which can enable powerful tests even in scenarios where martingale-based methods lack power \citep{ramdas22testing}.

We denote as $\X^\infty$ the space of sequences $(x_t)_{t\geq 1}\subseteq\X$. For convenience, we will often write $x^\infty$ for $(x_t)_{t\geq 1}$. We endow $\X^\infty$ with the product sigma-field generated by cylinder sets, and we denote as $\PR_{\X^\infty}$ the space of probability measures on $\X^\infty$. For a fixed $\mu\in(0,1)$, in this section we  focus on the hypothesis class
$$\Hh_\mu^\infty = \big\{P\in\PR_{\X^\infty}\,:\,\E_P[X_t|X^{t-1}]=\mu\,,\;\forall t\geq 1\big\}\,,\footnote{As usual $\E_P[X_t|X^{t-1}]=\mu$ holds $P$-almost everywhere.}$$
where again we use the convention $\E_P[X_1|X^0] = \E_P[X_1]$. 
This time, we will not aim to study the e-variables for $\Hh^\infty_\mu$, as these would be functions that take a whole sequence as argument. Conversely, we are interested in tools that allow us to stop the test if we think enough data-points have been observed to decide whether or not the hypothesis has to be rejected. To make this rigorous we first need some definitions. First, we define a \emph{finite stopping time} as a measurable\footnote{Here we endow $\Nn$ with the discrete sigma-field.} function $\tau:\X^\infty\to\Nn$ such that, for any $t\geq 0$ the set $\{\tau = t\}$ is measurable with respect to the sigma-field $\F_t$, generated by the projection on the first $t$ components of the sequences in $\X^\infty$. We let $\T$ denote the set of all finite stopping times. 
\begin{definition}
Let $E = (E_t)_{t\geq 0}$ be  a sequence  of non-negative Borel functions $E_t:\X^t\to[0,+\infty)$. We say that $E$ is an \emph{e-process} (for $\Hh_\mu^\infty$) when, for any $P\in\Hh^\infty_\mu$ and any $\tau\in\T$, $\E_P[E_\tau]\leq 1$.
\end{definition}
Sequential testing using e-processes proceeds as follows (see \citealt{ramdas2024hypothesis} for more details). Before observing any data, an e-process $E = (E_t)_{t\geq 0}$ for $\Hh_\mu^\infty$ is fixed.\footnote{Note that although the e-process  is fixed before observing any data, each $E_t$ is a function of the past observations $x^{t-1}$, allowing the test to adapt to the data as they become available.}
Then, at each round $t$, the player's wealth is defined as $\log E_t(x^t)$. $\Hh_\mu^\infty$ is rejected  if this wealth ever exceeds $\log  1/\delta$. The definition of an e-process ensures that this yields a sequential test with Type I error controlled at level $\delta$. We remark that this testing procedure extends the testing-by-betting framework of \Cref{def:game} to the more expressive setting where the null hypothesis takes into account the entire dependence structure of the data sequence. 

We denote as $\Ee_\mu^{\infty}$ the set of all e-processes for $\Hh_\mu^\infty$. An \emph{e-process class} is any subset of $\Ee_\mu^{\infty}$. Similarly to what we have done for the e-variables, we say that the e-process $E$ \emph{majorises} the e-process $E'$ (we write $E\succeq E'$) if for every $t\geq0$ we have $E_t\succeq E'_t$. If $E\succeq E'$ and there is a $t\geq0$ such that $E_t\succ E'_t$, then $E$ \emph{strictly} majorises $E'$, and we write $E\succ E'$. We call an e-process  \emph{maximal} if it is not strictly majorised by any other e-process. We say that an e-process class $\Ee$ is \emph{majorising} if, for every $E\in\Ee^{\infty}_\mu$, there is $E'\in\Ee$ such that $E\preceq E'$. We say that $\Ee$ is \emph{optimal} if it is majorising and all its elements are maximal. We remark that if an optimal majorising e-process class exists, then it is unique, it consists with the set of all the maximal e-processes, and it is included in every other majorising e-process class.

We let $\Lambda_\mu^\infty$ be the set of all sequences $\lambda^\infty = (\lambda_t)_{t\geq 1}$, with $\lambda_t:\X^{t-1}\to I_\mu$ Borel for all $t\geq 1$. For any $\lambda^\infty\in\Lambda_\mu^\infty$, we define $E^{\lambda^\infty} = (E^{\lambda^\infty}_T)_{T\geq 0}$ via $$E^{\lambda^\infty}_T(x^T) = E_{\lambda^T}(x^T) = \prod_{t=1}^T\big(1+\lambda_t(x^{t-1})(x_t-\mu)\big)\,,$$ for $T\geq 1$, and $E_0=1$. These are the sequences associated with the wealth of a coin-betting player. As $E^{\lambda^\infty}$ defines a martingale under $\Hh_\mu^\infty$, it is also an e-process for such hypothesis. We hence define the \emph{coin-betting e-process class} $$\Ecbinf_\mu = \{E^{\lambda^\infty}\,:\,\lambda^\infty\in\Lambda^\infty_\mu\}\,.$$ 

The next result shows that  the coin-betting approach is optimal at the level of e-processes for $\Hh_\mu^\infty$. The proof (see \ref{app:proc}) follows the ideas introduced in the proofs of \Cref{thm:main,thm:multi}.

\begin{theorem}\label{thm:proc}
    For any $\mu\in(0,1)$ the coin-betting e-process class is the optimal e-process class for $\Hh_\mu^\infty$.
\end{theorem}
\section{Perspectives}\label{sec:perspectives}
\Cref{thm:main} gives a rigorous sense to the claim that the coin-betting formulation is optimal among the e-variable-based approaches to test the mean and build confidence sequences given a sequence of independent draws from an unknown supported on $[0,1]$. \Cref{thm:multi,thm:proc}  extend this optimality result when testing the hypothesis of a fixed conditional mean. To formalise these claims, we introduced the notions of majorising and optimal e-classes, which may be of independent interest in the context of sequential testing, beyond the scope of this paper. The main novelty of these concepts lies in defining ``optimality'' in terms of e-classes, rather than individual e-variables. This perspective contrasts with the notion of \emph{log-optimality} of a single e-variable with respect to an alternative hypothesis, widely discussed in the literature (e.g., \citealp{koolen2022log,grunwald2024safe, larsson2024numeraire}). Notably, log-optimality is defined in terms of an alternative hypothesis, against which the null is tested. In such setting there can be indeed a single (up to null sets under the alternative hypothesis) \emph{best} e-variable. However, here we adopt a different perspective, where no alternative is defined. It is not hard to see that if the optimal e-class exists, whenever one considers an alternative that allows to define a log-optimal e-variable, a ``version''\footnote{As the log-optimal e-variable is defined up to null sets under the alternative, there might be elements of this family of e-variables that are not maximal. However, there is always a maximal e-variable among them if the optimal e-class exists.} of this e-variable must lie in the optimal e-class. Interestingly, in the case of the coin-betting e-class, for each $E\in\Ecb_\mu$ one can find an alternative hypothesis such that $E$ is log-optimal.\footnote{Specifically, Theorem 1 in \cite{grunwald2024safe} implies that, for $\lambda\in I_\mu$, $E_\lambda$ is log-optimal for the point alternative $\{Q\}$, with $Q = (1-\mu)E(0)\delta_0 + \mu E(1)\delta_1$, since $E_\lambda = \dd Q/\dd P$ on the support of $Q$, where $P = (1-\mu)\delta_0 + \mu\delta_1$ is in $\Hh_\mu$.}

The concept of maximality for e-variables, as introduced in this paper, is essentially equivalent to the classical statistical notion of admissibility. An e-variable is considered maximal if no other e-variable strictly dominates it. Likewise, the idea of a majorising e-class corresponds to the notion of a complete class, with the optimal e-class representing the minimal complete class of e-variables. In this sense, the problem of identifying the optimal e-class can be seen as an instance of characterising the minimal complete class of tests, a question largely studied in classical statistics. For a detailed discussion of these ideas in traditional statistics we refer to \cite{lehmann2005testing}. In this work, however, we adopt terminology from partially ordered set (poset) theory to highlight the fact that with e-variables these properties follows from the standard dominance partial ordering among functions. 

The concept of admissibility in the context of e-variables and e-processes was previously introduced by \cite{ramdas2022admissible}, whose definition is closely aligned with our notion of maximality. However, a key distinction lies in the type of dominance considered: their framework often relies on almost sure dominance, whereas our definition of maximality requires everywhere dominance. This stricter requirement is motivated by the will of developing a theory that remains valid even in the absence of a known alternative, where any value in $\X$ might occur at the next observation. Although \cite{ramdas2022admissible} provide some necessary and sufficient conditions for admissibility, we emphasise that their results are insufficient to establish the optimality of the coin-betting e-class. Specifically, their necessary conditions assume the existence of a reference measure for the considered hypothesis,\footnote{A hypothesis $\Hh$ admits a "reference measure" if there exists a Borel measure $Q$ such that $P \ll Q$ for all $P \in \Hh$.} a requirement not satisfied by $\Hh_\mu$. More precisely, their findings imply that each individual e-variable in the coin-betting e-class is maximal, but not that $\Ecb_\mu$ forms a majorising e-class. This distinction, though subtle, is crucial: proving the optimality of coin-betting requires analyzing the collective properties of a set of e-variables, rather than evaluating them in isolation. It is precisely the fact that $\Ecb_\mu$ is a majorising e-class that guarantees that nothing is lost by relying on the coin-betting approach. 

To the author's knowledge, this is the first work to rigorously examine the optimality of the coin-betting formulation, as previous discussion on the topic has relied solely on heuristic arguments. For instance, \cite{waudbysmith23estimating} justify restricting to the coin-betting e-class by noting that $\Ecb_\mu$ is precisely the set $\bar\Ee_{\Hh_\mu}$ of e-variables whose expectation equals $1$ under every $P\in\Hh_\mu$. However, in general this property (often called \emph{exactness}) merely implies that all the elements in the e-class are maximal, and not its optimality. As a simple counterexample, consider $\Hh'_\mu=\{P\in\PR\,:\,\langle P, X\rangle\leq\mu\}$. Then, $\bar\Ee_{\Hh'_\mu}=\{1\}$, while the optimal e-class exists and consists of all functions in the form $E_\beta:x\mapsto1+\beta(x-\mu)$, with $\beta\in[0, \mu^{-1}]$ (see \citealp{clerico24optimal}). Although one might argue that for each maximal e-variable there is at least one measure that brings the expectation to $1$ (\citealt{grunwald2024beyond} calls such property \emph{sharpness}), this feature alone is not enough to ensure maximality ($\one{\mu}$, the function equal to $1$ on $\mu$ and $0$ elsewhere, is a sharp e-variable for $\Hh_\mu$, but it is clearly not maximal). 

Two papers \citep{clerico24optimal, larsson2025evariables} have appeared after the first preprint of this work, both  \emph{characterising} single-round e-variables for hypotheses defined via linear constraints, a framework that includes the testing for the mean via single-round e-variables as a special case. Their results directly imply our \Cref{thm:main} (but not \Cref{thm:multi,thm:proc})\footnote{Interestingly, one could see the multi-round hypothesis $\Hh_\mu^T$ as a linearly constrained hypothesis, in the sense of \cite{larsson2025evariables}. However, it does not seem to be trivial to directly derive \Cref{thm:multi} from their results.}. However, their analyses rely on considerably more abstract and technically advanced proof techniques. \cite{clerico24optimal} first proves results for finite domains and then extends them to the uncountable case by compactness and density arguments, while \cite{larsson2024numeraire} leverages and extends powerful duality tools from the theory of functional lattices. Conversely, a key strength of the present paper is to provide a simple and neat argument (illustrated in \Cref{fig:mu}) that works well in the simple setting considered.  Interestingly, this same argument is also at the base of the proofs of \Cref{thm:multi} and \Cref{thm:proc}. We  remark that a similar strategy was adopted in the proof of Lemma 2 in \cite{wang2025only}. 

Another interesting connection with this paper's approach and the literature is recent work on admissible merging for e-variables \citep{vovk24merging, wang2025only}.  A thorough exploration of these connections is an interesting avenue for future research. In particular, \cite{wang2025only} implies that the coin-betting single-round e-variables cannot be majorised by any e-variable that is a monotonic function. However, the current paper approach relies on this monotonicity assumption, and cannot hence imply directly \Cref{thm:main}. On the other hand,
\cite{vovk24merging} considered conditional sequential hypotheses, of which conditional mean testing is a special case. Although their framework adopts a slightly different perspective, some of their results may imply, or be equivalent to, a  weaker version of \Cref{thm:proc}, where a finite time horizon $T$ is fixed and the e-processes are defined as vectors $(E_t)_{t\in[0:T]}$, rather than sequences. As our proof techniques follow a different route compared to these two works, we believe our findings offer an alternative perspective that contribute to a broader understanding of the problem and could potentially help generalise or strengthen their results.

To conclude, we remark that the present paper does not aim to delve into the design of effective coin-betting strategies. For this, we refer the interested reader to the thorough analysis and discussion by \cite{orabona2023tight} and \cite{waudbysmith23estimating}. However, let us stress once more that characterising the optimal e-class simplifies the taks of designing such strageties, by provides the minimal set where e-values shall be picked. We discussed at the end of \Cref{sec:conf} an explicit example where this can be turned into a practical advantage, namely when using a Bayesian aggregation strategy that places a prior over the e-class considered. The tractable, parametric form of the coin-betting e-class makes such approach feasible, while its optimality ensures that no statistical power is lost.

As a final comment, although we have focused on random variables taking values in $\X\subseteq [0,1]$ for clarity of exposition, most results presented can be easily extended to any bounded closed real set. 

\section*{Acknowledgements} 
The author would like to thank Gergely Neu, Peter D.~Gr\"unwald, Nishant A.~Mehta, Hamish E.~Flynn, and Claudia M.~Chanu, for the insightful discussions that inspired this work. I would also like to thank Nick W.~Koning and Aditya Ramdas for their constructive feedback, which helped improve this work following its initial appearance, and Wouter M.~Koolen for helpful discussions on e-processes. Finally, I thank the anonymous reviewers for the insightful comments and feedback that helped improving this work. GPT-4o was used during the redaction of this paper to polish the presentation. All AI-generated text was reviewed and edited by the author, who takes full responsibility for the content of this manuscript. This project was funded by the European Research Council (ERC), under the European Union’s Horizon 2020 research and innovation programme (grant agreement 950180). 
\bibliography{bib}

\newpage
\appendix
\section{}
\subsection[Proof of Proposition 2]{Proof of \Cref{prop:hoef}}\label{app:hoef}
\begin{proof}
For the first claim, without loss of generality we let $\X=[0,1]$, which implis the desired result for any  $\X$.  For any $\alpha\in\R$, let $S_\alpha:x\mapsto \EHa(0) + (\EHa(1)-\EHa(0))x$ be the straight line that intersects $\EHa$ at $x=0$ and $x=1$. Since $\EHa$ is a convex function on $[0,1]$, we have  $\EHa(x)\leq S_\alpha(x)$ for all $x\in[0,1]$. Let us show that $E_{\lambda_\alpha}(x)\geq S_\alpha(x)$ for all $x$. As $E_{\lambda_\alpha}$ and $S_\alpha$ represents parallel straight lines, it suffices to show inequality for a single point, say $x=\mu$ where $E_{\lambda_\alpha}$ equals $1$. 
Hence, it is enough to prove that $1+(e^\alpha-1)\mu \leq e^{\alpha\mu+\alpha^2/8}$. As $\mu\in(0,1)$, we always have $1+(e^\alpha-1)\mu>0$, and we can define $u:\R\to\R$ as
$$u(\alpha) = \frac{\alpha^2}{8}+\mu\alpha - \log\big(1+\mu(e^\alpha-1)\big)\,.$$
Clearly, if  $u$ is non-negative, then $1+(e^\alpha-1)\mu \leq e^{\alpha\mu+\alpha^2/8}$ for all $\alpha\in\R$, and so we obtain the desired claim. First, note that $u$ is twice differentiable, and we can explicitly compute its first and second derivatives:
$$    u'(\alpha) = \frac{\alpha}4+\mu-\frac{\mu e^\alpha}{1+\mu(e^\alpha-1)}\quad\text{ and }\quad
    u''(\alpha) = \frac{1}{4} -\frac{\mu e^\alpha}{1+\mu(e^\alpha-1)}\left(1-\frac{\mu e^\alpha}{1+\mu(e^\alpha-1)}\right)\,.
$$
For any $\xi\in\R$ we have $\xi(1-\xi)\leq 1/4$, so $u''(\alpha)\geq 0$ for all $\alpha$, and $u$ is convex. Moreover $u'(0) = 0$ and $u(0)=0$, so that $0$ is the minimum of $u$, which must then be non-negative.

For the second claim, fix any $\lambda\neq 0$. If $\alpha\neq 0$, then $\EHa(\mu) = e^{-\alpha^2/8}<1 = E_\lambda(\mu)$. If $\alpha=0$, then $\EHa$ is identically equal to $1$, and $\max(E_\lambda(0),E_\lambda(1))>1$, so we conclude.\qed
\end{proof}
\subsection[Characterising the e-variables for hat H2(1/2) with X=(0,1/2,1)]{Characterising the e-variables for $\hat\Hh_{1/2}^2$ with $\X=\{0,1/2,1\}$}\label{app:iid}
Let $\X=\{0,1/2,1\}$, and $\mu=1/2$. Let $Q\in\Hh_{1/2}^1$ and let $q=Q(\{1/2\})$. The mean constraint reads $Q(\{1\}) + q/2 = 1/2$, yielding $Q(\{1\}) = (1-q)/2$. The normalisation yields $Q(\{0\}) = 1-q-(1-q)/2 = (1-q)/2$. This shows that there is a one-to-one correspondence between $[0,1]$ and $\Hh_{1/2}^1$. Now, since by definition there is a one-to-one correspondence between $\Hh_{1/2}^1$ and $\hat\Hh_{1/2}^2$, we obtain that we can parametrise $\hat\Hh_{1/2}^2$ by $[0,1]$. For $q\in[0,1]$, we hence let $P_q$ be the (unique) element of $\Hh_{1/2}^2$ that gives mass $q^2$ to $(1/2,1/2)$. 

Let $E:\X^2\to[0,+\infty)$ be a non-negative function. We have that
$$\E_{P_q}[E] = q^2\xi_0 + 2q(1-q)\xi_1+(1-q)^2\xi_2 = (\xi_0+\xi_2-2\xi_1)q^2 - 2(\xi_0-\xi_1)q + \xi_2\,,$$
where $\xi_0 = E(1/2,1/2)$, $\xi_1 = (E(1,1/2)+E(1/2,1)+E(1/2,0)+E(0,1/2))/4$, and $\xi_2 = (E(1,1)+E(1,0)+E(0,1)+E(0,0))/4$. Now, $E$ is in $\hat\Ee_{1/2}^2$ if, and only if,
$$\max_{q\in[0,1]}\big((\xi_0+\xi_2-2\xi_1)q^2-2(\xi_0-\xi_1)q+\xi_2\big)\leq 1\,.$$ In particular, checking for $q=0$ and $q=1$ implies that $\xi_0\leq 1$ and $\xi_2\leq 1$. If both $\xi_0$ and $\xi_2$ are equal to $1$, then the constraint on $\xi_1$ becomes $\max_{q\in[0,1]}(1-\xi_1)(q^2 - q)\leq 0$, which implies $\xi_1\leq 1$. We are left to check whether $\xi_1$ can be larger than $1$ when at least one among $\xi_0$ and $\xi_2$ is strictly smaller than $1$. In such case, we note that we have $\xi_2-\xi_1<0$ and $\xi_0-\xi_1<0$. This implies that the parabola is concave and achieve its maximum at $q^\star = (\xi_0-\xi_1)/(\xi_0-\xi_1 + \xi_2-\xi_1)\in(0,1)$. The maximum is equal to $\xi_2(\xi_1-\xi_0)^2/(2\xi_1-\xi_0-\xi_2)$. Asking that this quantity is less than one reduces to the constraint $\xi_1\leq1+\sqrt{(1-\xi_0)(1-\xi_2)}$.

\subsection[Proof of Theorem 2]{Proof of \Cref{thm:multi}}\label{app:multi}
\begin{proof}
    First let us show that all the e-variables in $\EcbT_\mu$ are maximal. This can be proved essentially with the same argument we used for \Cref{thm:main}. Specifically, fix $E\in\mathcal E^{\mathrm{cb},T}_\mu$ and let $E'\succeq E$ be an e-variable for $\Hh_\mu^T$. Fix any $x^T\in\X^T$. It is easy to see that there is $P_{x^T}\in\Hh_\mu^T$ such that $P_{x^T}(\{x^T\})>0$. Then, $0\leq P_{x^T}(\{x^T\})( E'(x^T)-E(x^T) )\leq \E_{P_{x^T}}[E'-E] = \E_{P_{x^T}}[E']-1\leq 0$. Hence, $E'(x^T) = E(x^T)$, so $E=E'$, since $x^T$ was arbitrary, and $E$ is maximal. 

    The fact that $\EcbT_\mu$ is a majorising e-class follows directly from the more general \Cref{thm:proc}. Indeed, let $E\in\Ee_\mu^T$. Then we can consider an e-process $\tilde E = (\tilde E_t)_{t\geq 0}\in\Ee_\mu^\infty$, with $\tilde E_T = E$, and $\tilde E_t = 0$ for $t\neq 0$. By \Cref{thm:proc}, there is $\lambda^\infty$ such that $E^{\lambda^\infty}\succeq \tilde E$. In particular $E = \tilde E_T \preceq E_{\lambda^T}\in\EcbT_\mu$, which show that $\EcbT_\mu$ is a majorising e-class.  However, since the general proof of \Cref{thm:proc} is rather technical, we provide a detailed argument for \Cref{thm:multi} in the case $T=2$ as a simplified example, which might help build intuition for the proof of \Cref{thm:proc}.

    As we have already shown that all the e-variables in $\mathcal E^{\mathrm{cb},2}$ are maximal, we are left with checking that $\mathcal E^{\mathrm{cb},2}$ is a majorising e-class. So, fix an e-variable $E$ for $\Hh_\mu^2$. We need to show that there is a $\lambda^2\in\Lambda^2_\mu$ such that $E_{\lambda^2}\succeq E$. We start by showing that there is a $\lambda_1\in I_\mu$ such that, for any $Q\in\Hh^1_\mu$ and $x\in\X$, $$\E_Q[E(x, X)]\leq 1+\lambda_1(x-\mu)\,.$$ For this we will essentially use the idea that was at the core of the proof of \Cref{thm:main}. For any $Q$ and $x\geq\mu$, we can define $V_{x,Q} = P_x\otimes Q$, where $P_x = \frac{\mu}{x}\delta_{x} + (1-\frac{\mu}{x})\delta_0$. Then, $V_{x,Q}\in\Hh_{\mu}^2$. We have that $1\geq \E_{V_{x,Q}}[E]= \tfrac{\mu}{x}\E_Q[E(x, X)]+ \big(1-\tfrac{\mu}{x}\big)\E_Q[E(0, X)]$, and so $\E_Q[E(x, X)]\leq x/\mu = F_\mu(x)$,  with $F_\mu$  as in \Cref{lemma:F}. A similar argument can be used to show that $\E_Q[E(x, X)]\leq F_\mu(x)$ is true also when $x\leq\mu$. In particular, the following two sets are non-empty:
    \begin{align*}
        &B_0 = \big\{\beta\in I_\mu\,:\,\forall x\in\X\cap[0,\mu)\,,\,\forall Q\in\Hh_\mu^{1}\,,\,\E_Q[E(x,X)]\leq 1+\beta (x-\mu)\big\}\,;\\
        &B_1 = \big\{\beta\in I_\mu\,:\,\forall x\in\X\cap(\mu,1]\,,\,\forall Q\in\Hh_\mu^{1}\,,\,\E_Q[E(x,X)]\leq 1+\beta (x-\mu)\big\}\,.
    \end{align*} 
    $B_0$ and $B_1$ are intersections of closed intervals (one per each allowed $x$ and $Q$), and as such they must be closed intervals. We can find $\beta_0$ and $\beta_1$ such that 
    $B_0 = [(\mu-1)^{-1},\beta_0]$ and $B_1 = [\beta_1, \mu^{-1}]$. We will show that $\beta_0\geq\beta_1$, and hence that $B_0\cap B_1\neq\varnothing$. The argument is essentially the same one that was depicted in \Cref{fig:mu}. Assume that $\beta_0<\beta_1$, and let $\beta^\star\in(\beta_0,\beta_1)$. As $\beta^\star\notin B_0$, there is $Q_0\in\Hh^{1}_\mu$ and $u_0<\mu$ in $\X$ such that $\E_{Q_0}[E(u_0, X)]> 1+\beta^\star (u_0-\mu)$. Similarly, there must be $Q_1\in\Hh_\mu^{1}$ and $u_1>\mu$ in $\X$, such that $\E_{Q_1}[E(u_1, X)]> 1 + \beta^\star (u_1-\mu)$, since $\beta^\star\notin B_1$. Let $V = \frac{u_1-\mu}{u_1-u_0}\delta_{u_0}\otimes Q_0 + \frac{\mu-u_0}{u_1-u_0}\delta_{u_1}\otimes Q_1$. Then, $ V\in\Hh^2_\mu$. Moreover,
    $$\E_{V}[E] >\tfrac{u_1-\mu}{u_1-u_0}\big(1+\beta^\star (u_0-\mu)\big) + \tfrac{\mu-u_0}{u_1-u_0}\big(1+\beta^\star (u_1-\mu)\big)=1\,,$$ which is a contradiction since $E\in\Ee^2_\mu$. We have thus established that $B_0\cap B_1$ is non-empty, and in particular there is $\lambda_1\in I_\mu$ such that, for every $x\in\X$ and every $Q\in\Hh_\mu^{1}$, $\E_Q[E(x, X)]\leq 1 + \lambda_1 (x-\mu)$.
 
Now, fix any $x\in\X$ such that $1+\lambda_1(x-\mu)\neq 0$ and define the function $E^x:\X\to\R$ as $$E^x:y\mapsto \frac{E(x,y)}{1+\lambda_1(x-\mu)}\,.$$ Clearly $E^x$ is non-negative, and what we have shown above implies that $\E_Q[E^x]\leq 1$ for every $Q\in\Hh^1_\mu$. So, $E^x\in\Ee_\mu^1$, and there is $\hat\lambda_2^x\in I_\mu$ such that $$E(x,y)\leq \big(1+\lambda_1(x-\mu)\big)\big(1+\hat\lambda_2^x(y-\mu)\big)$$ for every $y\in\X$. On the other hand, if $x\in\X$ is such that $1+\lambda_1(x-\mu) = 0$, we have that, for every $Q\in\Hh_\mu^1$, $\E_Q[E(x,X)] = 0$. Clearly, this implies that $E(x,y)=0$ for every $y\in\X$. Hence, we have that, for every $(x,y)\in\X^2$, $$E(x, y)\leq \big(1+\lambda_1(x-\mu)\big)\big(1+\hat\lambda_2(x)(y-\mu)\big)\,,$$ where we defined $\hat\lambda_2:\X\to I_\mu$ as $\hat\lambda_2(x) = \hat\lambda_2^x$ if $1+\lambda_1(x-\mu)\neq 0$, and $\hat\lambda_2(x)=0$ otherwise.

Although the above inequality looks exactly like what we are looking for, we are not done yet, as nothing ensures that $\hat\lambda_2$ is a Borel function. To show the existence of a $\lambda_2:\X\to I_\mu$ that is Borel and such that $E\preceq E_{\lambda^2}$ we will use a functional separation theorem that can be derived from Lusin's separation theorem (see \ref{app:lusin}). First, let us define $\mathcal S=\{x\in\X\,:\,1+\lambda_1(x-\mu)\neq 0\}$, which is clearly a Borel set. We define the functions $u$ and $l$, from $\mathcal S$ to $\R$, as 
    \begin{align*}
    	u(x) &=\sup_{y\in\X\cap(\mu,1]}\frac{1}{y-\mu}\left(\frac{E(x,y)}{1+\lambda_1(x-\mu)}-1\right)\,;\\
   	l(x) &= \inf_{y\in\X\cap[0,\mu)}\frac{1}{\mu-y}\left(1-\frac{E(x,y)}{1+\lambda_1(x-\mu)}\right)\,.
    \end{align*}
    It is straightforward to check that, for any $x\in\mathcal S$, 
    $$(\mu-1)^{-1}\leq u(x)\leq \hat\lambda_2(x)\leq l(x)\leq \mu^{-1}\,.$$   
    Now, $u$ is the supremum on $y$ of the Borel mapping $(x,y)\mapsto\frac1{y-\mu}\big(\frac{E(x,y)}{1+\lambda_1(x-\mu)}-1\big)$. Hence, it is upper semi-analytic by \Cref{lemma:supinf}. Similarly, $l$ is lower semi-analytic as an infimum. In particular, by \Cref{prop:ufl}, there is a Borel mapping $\lambda_2:\mathcal S\to I_\mu$ such that $$u(x)\leq \lambda_2(x)\leq l(x)$$ for all $x\in\mathcal S$. We can extend $\lambda_2$ to the whole $\X$ by letting $\lambda_2(x)=0$ if $x\in\X\setminus\mathcal S$. This extension is still a Borel function since $\mathcal S$ is Borel. It is now straightforward to check that, for any $(x,y)\in\X^2$ we have
    $$E(x,y)\leq \big(1+\lambda_1(x-\mu)\big)\big(1+\lambda_2(x)(y-\mu)\big) = E_{\lambda^2}(x,y)\,,$$
    which concludes the proof for the case $T=2$.\qed
\end{proof}

\subsection[Proof of Theorem 3]{Proof of \Cref{thm:proc}}\label{app:proc}
The proof of \Cref{thm:proc} involves a few technicalities, mainly in order to solve issues linked to Borel measurability. To handle this we will consider ``simplified'' versions of $\Hh_\mu$, $\Hh_\mu^T$, and $\Hh_\mu^\infty$.
\subsubsection{Coarsening of the hypothesis}\label{app:coarse}
First, we define a coarser version of $\Hh_\mu$. More precisely, we let $\bar\Hh_\mu$ denote the set of probability measures in $\Hh_\mu$ whose support has at most two elements. We define as $\D$ the set 
$$\D = \{(a,a')\in\X^2\,:\, \mu\in[a,a']\text{ or }\mu\in[a',a]\}\,.$$
Define the non-negative function $W:\D\to[0,1]$ as 
\begin{equation}\label{eq:W}W(a_1, a_2) = \frac{a_2-\mu}{a_2-a_1}\,,\end{equation} here adopting the convention that $0/0=1$.
For $d=(a_1,a_2)\in\D$, there is exactly one measure $Q_{d}\in\bar\Hh_\mu$ that has support in $\{a_1,a_2\}$. This is 
$$Q_d = W(a_1, a_2)\delta_{a_1} + (1-W(a_1, a_2))\delta_{a_2}\,.$$ Clearly, the mapping $d\mapsto Q_d$ is surjective on $\bar\Hh_\mu$. Moreover, for any Borel $f:\X\to\R$, it is straighforward to check that the mapping $d\mapsto \E_{Q_d}[f]$ is Borel from $\D$ to $\R$.

We can extend these ideas to sequential hypotheses. We define $\bar\Hh_\mu^T$ as a subset of $\Hh_\mu^T$ consisting of measures that are built by iterating over $T$ time steps the 2-point construction that we used for $\bar\Hh_\mu$. More precisely, $Q \in\Hh_\mu$ is in $ \bar\Hh_\mu^T$ if for $X^T\sim Q$ the marginal $X_1$ is in $\bar\Hh_\mu$, and, for every $t \in [2:T]$, the conditional distribution of $X_t$ under $Q$, given $X^{t-1}$, belongs to $\bar\Hh_\mu$ for $Q$-almost every $X^{t-1}$. In particular, the support of $Q$ lies on at most $2^T$ trajectories in $\X^T$, and each branching step preserves the conditional mean constraint $\E_Q[X_t| X^{t-1}] = \mu$. This defines a  subset of $\Hh_\mu^T$ where every conditional law is supported on at most two points. We can of course extend the whole construction to sequences, and again consider the set $\bar\Hh_\mu^\infty\subseteq\Hh_\mu^\infty$ of the sequences such that $X_t$, conditioned on $X^{t-1}$, has mean $\mu$ and support on at most two elements. 

To make things more explicit, consider the case $T = 3$. Fix $(a_1, a_2) \in \D$, $b^4 =(b_1,b_2, b_3, b_4) \in \D^2$ (namely, $(b_1,b_2)\in\D$ and $(b_3,b_4)\in\D$), and $c^8 \in \D^4$. Then, given this tuple $d=(a^2, b^4, c^8) \in \D^{7}$, we can define a measure $Q_d \in \bar\Hh_\mu^3$ supported on the eight leaves of the tree shown in \Cref{fig:tree}. Each level of the tree corresponds to a time step. At the root, we begin by choosing among $a_1$ and $a_2$, by determining the value of $X_1$ (whose support under $Q_d$ is in $\{a_1, a_2\}$). Then, depending on whether we got $a_1$ or $a_2$, we branch using $(b_1, b_2)$ or $(b_3, b_4)$, which gives $X_2$. Finally, again depending on the previous choices, $Q_d$ gives different options for $X_3$, encoded in the branches leading to the leaves. The weights that $Q_d$ gives to each leaf is univocally determined by the constraint that the conditional means have to be equal to $\mu$. For instance, the mass that $Q_d$ assigns to the point $(a_1, b_2, c_3)$ is given by
$$
Q_d(\{a_1, b_2, c_3\}) = W(a_1,a_2) \times (1-W(b_1,b_2)) \times W(c_3,c_4)\,.
$$
Proceeding in this way, we can build a surjection from $\D^7$ to $\bar\Hh_\mu^3$.
\begin{figure}[t!]
        \centering
        \begin{tikzpicture}[
  grow=right,
  sibling distance=8mm,
  level distance=20mm,
  edge from parent/.style={draw}
]
  \node {$\bullet$}
    child { node {$a_2$}
      child { node {$a_2, b_4$}
        child { node {$a_2, b_4, c_8$} }
        child { node {$a_2, b_4, c_7$} }
      }
      child [missing]
      child { node {$a_2, b_3$}
        child { node {$a_2, b_3, c_6$} }
        child { node {$a_2, b_3, c_5$} }
      }
    }
    child [missing]
    child [missing]
    child [missing]
    child { node {$a_1$}
      child { node {$a_1, b_2$}
        child { node {$a_1, b_2, c_4$} }
        child { node {$a_1, b_2, c_3$} }
      }
      child [missing]
      child { node {$a_1, b_1$}
        child { node {$a_1, b_1, c_2$} }
        child { node {$a_1, b_1, c_1$} }
      }
    };
\end{tikzpicture}
        \caption{Tree representation associated to $Q_d$ for $T=3$, where $d=(a^2, b^4, c^8)$.}
        \label{fig:tree}
    \end{figure}

More generally, we can proceed analogously and see that every $d\in\D^{2^T-1}$ defines a unique $Q_d\in\bar\Hh_\mu^T$. For convenience, we henceforth denote $\D^{2^T-1}$ as $\D_T$. We have hence constructed a surjection from $\D_T$ to $\bar\Hh^T_\mu$, mapping $d$ to $Q_d$. It is clear from its definition that $\D_T$ is a Borel subset of $\X^{2(2^{T} - 1)}$. Moreover, it is easy to check that, for any fixed Borel function $f : \X^T \to \mathbb{R}$, $d \mapsto \mathbb{E}_{Q_d}[f]$ is a Borel measurable map from $\D_T$ to $\R$, which directly follows from the Borel measurability of $W$.

We remark that, by definition, $\bar\Hh_\mu\subseteq\Hh_\mu$. In particular, all the e-variables for $\Hh_\mu$ are also e-variables for $\bar\Hh_\mu$. An equivalent conclusion holds for the e-variables for $\Hh_\mu^T$ and the e-processes for $\Hh_\mu^\infty$. Thus, if we show that $\Ecbinf_\mu$ is a majorising e-process class for $\bar\Hh_\mu^\infty$, this will automatically imply that it is a majorising e-process class for $\Hh_\mu^\infty$. This is indeed the strategy that we will follow in the proof of \Cref{thm:proc}. The technical reason that required us to introduce this coarsening of the hypothesis, is that to deal with Borel measurability we will follow an approach similar to what done in the proof of \Cref{thm:multi} for $T=2$. In particular, we will apply \Cref{prop:ufl}, which tells us that we can always find a Borel function sandwiched between an upper and a lower semi-analytic ones (see \ref{app:lusin}). To follow this route, we will need the following technical result, whose proof is deferred to \ref{app:trees}.
\begin{lemma}\label{lemma:trees}
Let $f = (f_s)_{s\geq 1}$ denote a sequence of bounded Borel functions, with $f_s:\X^s\to\R$. Fix $t\geq 1$ and, for any $x^t\in\X^t$ and $s\geq0$, let $f_s^{x^t}:\X^s\to\R$ be defined via $f_s^{x^t}(y^s) = f_{t+s}(x^t,y^s)$. Define the functions $u$ and $l$, from $\X^t\to\R$, as
$$
    u(x^t) = \sup_{Q\in\bar\Hh^\infty_\mu}\sup_{\tau\in\T}\E_Q[f_\tau^{x^t}]\quad\text{ and }\quad
    l(x^t) = \inf_{Q\in\bar\Hh^\infty_\mu}\inf_{\tau\in\T}\E_Q[f_\tau^{x^t}]\,.
$$
Then, $u$ is upper semi-analytic and $l$ is lower semi-analytic.
\end{lemma}
\subsubsection{Proof of the theorem}
We start by a preliminary lemma, whose proof follows closely the argument we used to prove \Cref{thm:main}.
\begin{lemma}\label{lemma:TQx}
    Let $\mu\in(0,1)$ and $E$ be an e-process for $\bar\Hh^\infty_\mu$. Then there is $\lambda_1\in I_\mu$ such that, for any $\tau\in\T$, any $Q\in\bar\Hh_\mu^\infty$, and any $x\in\X$, we have that
	$$\E_Q[E^x_\tau]\leq E_{\lambda_1}(x) = 1 + \lambda_1 (x-\mu)\,,$$
    where $E^x = (E^x_t)_{t\geq 0}$ is the sequence of Borel functions $E_t^x:\X^{t}\to I_\mu$, defined via $E_t^x(y^{t}) = E_{t+1}(x, y^t)$, for $t\geq 1$ and $y^t\in\X^t$, and $E_0^x = E_1(x)$.
\end{lemma}
\begin{proof}
Define the sets
\begin{align*}
        &B_0 = \big\{\beta\in I_\mu\,:\,\forall x\in\X\cap[0,\mu),\forall Q\in\bar\Hh^\infty_\mu,\forall \tau\in\T,\, \E_Q[E^x_\tau]\leq 1 + \beta(x-\mu)\big\}\,;\\
        &B_1 = \big\{\beta\in I_\mu\,:\,\forall x\in\X\cap(\mu,1],\forall Q\in\bar\Hh^\infty_\mu,\forall \tau\in\T,\, \E_Q[E^x_\tau]\leq 1 + \beta(x-\mu)\big\}\,.
\end{align*}
We claim that $B_0$ and $B_1$ are non-empty. Indeed, we now show that $\mu^{-1}\in B_1$. Analogously one can prove that $(\mu-1)^{-1}\in B_0$. Fix $x\in\X\cap[\mu,1]$, $Q\in\bar\Hh_\mu^{\infty}$, and $\tau\in\T$. Define $\tau_x:\X^\infty\to\Nn$ as $\tau_x(x^\infty) = 1+\tau(x^{2:\infty})$, if $x_1=x$, otherwise $\tau_x(x^\infty) = 1$. It is easily checked that $\tau\in\T$. Also, we let $P_x = \frac{\mu}{x}\delta_x + (1-\frac{\mu}x)\delta_0$, and we define $V_{x,Q} = P_x\otimes Q$.\footnote{Here we used that, clearly, one can write $\X^\infty = \X\times\X^\infty$.} Clearly, $V_{x,Q}\in\bar\Hh_\mu^\infty$. Since $E$ is an e-process,
$$1\geq\E_{V_{x,Q}}[E_{\tau_x}] = (1-\tfrac{\mu}x)E_1(0) + \tfrac{\mu}x\E_Q[E^x_\tau]\geq \tfrac{\mu}x\E_Q[E^x_\tau]\,.$$
So, $E_Q[E^x_\tau]\leq \frac x\mu= 1 + \mu^{-1}(x-\mu)$. Thus, $\mu^{-1}\in B_1$. 

Once more, our next argument follows closely the one depicted in \Cref{fig:mu}. Since both $B_0$ and $B_1$ can be written as intersections of non-empty closed intervals, we can find $\beta_0$ and $\beta_1$ such that $B_0=[(\mu-1)^{-1}, \beta_0]$ and $B_1=[\beta_1,\mu^{-1}]$. As usual, we want to show that $B_0\cap B_1$ is non-empty, or equivalently that $\beta_1\leq\beta_0$. Assume that this was not the case and there is $\beta^\star\in(\beta_0,\beta_1)$. Since $\beta^\star\notin B_0$, there must be $u_0<\mu$ in $\X$, $Q_0\in\bar\Hh_\mu^{\infty}$, and $\tau_0\in\T$, such that $\E_{Q_0}[E^{u_0}_{\tau_0}]>1+\beta^\star(u_0-\mu)$. Similarly, as $\beta^\star\notin B_1$, one can find $u_1\in\X\cap(\mu,1]$, $Q_1\in\bar\Hh_{\mu}^{\infty}$, and $\tau_1\in\T$, such that $\E_{Q_1}[E^{u_1}_{\tau_1}]>1+\beta^\star(u_1-\mu)$. Define $\tau:\X^\infty\to\Nn$ via $\tau(x^\infty) = 1 + \tau_0(x^{2:\infty})$ if $x_1<\mu$, $\tau(x^\infty)= 1+\tau_1(x^{2:\infty})$ if $x_1>\mu$, and $\tau(x^\infty)=1$ if $x_1=\mu$. Also, let $V = \frac{u_1-\mu}{u_1-u_0}\delta_{u_0}\otimes Q_0 + \frac{\mu-u_0}{u_1-u_0}\delta_{u_1}\otimes Q_1$. Then, we have that
$$ \E_V[E_\tau] = \frac{u_1-\mu}{u_1-u_0}\,\E_{Q_0}[E^{u_0}_{\tau_0}] + \frac{\mu-u_0}{u_1-u_0}\,\E_{Q_1}[E^{u_1}_{\tau_1}] > 1\,. $$
However, this is a contradiction since $\tau\in\T$, $V\in\bar\Hh_\mu^\infty$, and $E$ is an e-process. We can thus conclude that $B_0\cap B_1\neq\varnothing$ and, in particular, there must be $\lambda_1\in I_\mu$ such that, for every $x\in\X\setminus\{\mu\}$, every $Q\in\bar\Hh_\mu^{\infty}$, and every $\tau\in\T$, we have $\E_Q[E^x_\tau]\leq 1+\lambda_1(x-\mu)$. 

To conclude, we only need to check the case $x=\mu$. However, the argument that we used at the beginning of this proof to show that $\mu^{-1}\in B_1$ holds for $x=\mu$, showing that $\E_Q[E_\tau^\mu]\leq 1$, for any $Q\in\bar\Hh_\mu^\infty$ and any $\tau\in\T$. So, we conclude.\qed
\end{proof}
We can now prove \Cref{thm:proc}. The main idea is a generalisation of the approach that we used to deal with the case $T=2$ in the proof of \Cref{thm:multi}.
\begin{proof}[of \Cref{thm:proc}]
	First, let us show that all the e-processes in $\Ecbinf_\mu$ are maximal (for $\Hh_\mu^\infty$). Assume that this was not the case. Then there is a $\lambda^\infty\in\Lambda_\mu^\infty$ and an e-process (for $\Hh_\mu^\infty$) $E\succeq E^{\lambda^\infty}$ such that, for some $t\geq t$, $E_t\succ E_{\lambda^t}$ (clearly this cannot happen for $t=0$, as we must have $E_0\leq 1$ for every e-process). So, there is $\hat x^t\in\X^t$ such that $E_t(\hat x^t)>E_{\lambda^t}(\hat x^t)$. It is not hard to see that there is a $Q\in\Hh_\mu^\infty$ that puts non-zero mass on the set $\{x^\infty\in\X^\infty\,:\,x^t = \hat x^t\}$. We can consider the constant stopping time $\tau=t$. Then, we have that $0< E_t(\hat x^t) - E_{\lambda^t}(\hat x^t) \leq \E_Q[E_t - E_{\lambda^t}] = \E_Q[E_\tau] - 1 \leq  0$,  a contradiction.	

	Hence, we are left with showing that $\Ecbinf_\mu$ is a majorising e-process class for $\Hh^\infty_\mu$. Since every e-process in for $\Hh^\infty_\mu$ is also an e-process for $\bar\Hh^\infty_\mu$, it is sufficient to show that $\Ecbinf_\mu$ is a majorising e-process class for $\bar\Hh^\infty_\mu$. Fix an e-process $E$ (for $\bar\Hh^\infty_\mu$). We introduce the following notation. For $t\geq 1$ and $x^t\in\X^t$, we denote as $E^{x^t} = (E^{x^t}_s)_{s\geq 0}$ the sequence of non-negative functions $E^{x^t}_s:\X^{s-1}\to I_\mu$, defined via $E_s^{x^t}(y^s) = E_{t+s}(x^t, y^s)$, for any $y^s\in\X^s$ and $s\geq 1$, and $E_0^{x^t} = E_t(x^t)$. In what follows, we say that a $T$-tuple $\lambda^{T}$ of Borel functions $\lambda_i:\X^{i-1}\mapsto I_\mu$, \emph{dominates $E$ at level $T$} if, for all $Q\in\bar\Hh^\infty_\mu$, for any $t\in[1:T]$, for all $\tau\in \T$ and $x^t\in\X^t$, we have that $$\E_Q[E^{x^t}_\tau]\leq E_{\lambda^t}(x^t)\,.$$ We will show that there is a $\lambda^\infty\in\Lambda_\mu^\infty$ such that, for any $T\geq 1$, $\lambda^T$ dominates $E$.

    We construct this sequence progressively. By \Cref{lemma:TQx}, we know that there is $\lambda_1\in I_\mu$ that dominates $E$ at level $1$. Now, say that, for some $T\geq 2$, there is a $\lambda^{T-1}\in\Lambda_\mu^{T-1}$ that dominates $E$ at level $T-1$. 
Let us show that this imply the existence of a Borel $\lambda_T:\X^{T-1}\to I_\mu$, such that $\lambda^T=(\lambda^{T-1},\lambda_T)$ dominates $E$ at level $T$. Define $\mathcal S=\{x^{T-1}\in\X^{T-1}\,:\,E_{\lambda^{T-1}}(x^{T-1})\neq 0\}$. For any $x^{T-1}\in\mathcal S$, define the sequence of Borel functions $\tilde E^{x^{T-1}}  = (\tilde E^{x^{T-1}}_t)_{t\geq 0}$ as follows. $\tilde E_0^{x^{T-1}} = 1$, and $\tilde E_t^{x^{T-1}}(y^t) = E_{T-1+t}(x^{T-1}, y^t)/E_{\lambda^{T-1}}(x^{T-1})$ for $t\geq 1$ and $y^t\in\X^t$. By construction, $\tilde E^{x^{T-1}}$ is an e-process. In particular, by \Cref{lemma:TQx}, there must be $\tilde\lambda_1^{x^{T-1}}\in I_\mu$ (which of course might depend on $x^{T-1}$ in a non-Borel way) dominating $\tilde E^{x^{T-1}}$ at level $1$. So, for any $x^{T-1}\in\mathcal S$ and $y\in\X$,
$$\sup_{Q\in\bar\Hh^\infty_\mu}\sup_{\tau\in\T}\E_Q[  E_\tau^{(x^{T-1},y)}] \leq  E_{\lambda^{T-1}}(x^{T-1}) \big(1 +  \tilde\lambda_1^{x^{T-1}}(y-\mu)\big)\,.$$ 
We now define the mappings $u$ and $l$ on $\mathcal S$ as 
\begin{align*}
	u(x^{T-1}) &=\sup_{\tau\in\T}\sup_{Q\in\bar\Hh^\infty_\mu}\sup_{y\in\X\cap(\mu,1]}\frac{1}{y-\mu}\left(\frac{\E_Q[E^{(x^{T-1},y)}_\tau]}{E_{\lambda^{T-1}}(x^{T-1})}-1\right)\,;\\
   	l(x^{T-1}) &= \inf_{\tau\in\T}\inf_{Q\in\bar\Hh^\infty_\mu}\inf_{y\in\X\cap[0,\mu)}\frac{1}{\mu-y}\left(\frac{\E_Q[E^{(x^{T-1},y)}_\tau]}{E_{\lambda^{T-1}}(x^{T-1})}-1\right)\,.
\end{align*}
It follows from \Cref{lemma:trees} that $(x^{T-1},y)\mapsto \sup_{\tau\in\T}\sup_{Q\in\bar\Hh^\infty_\mu}\frac{1}{y-\mu}\big(\frac{\E_Q[E^{(x^{T-1},y)}_\tau]}{E_{\lambda^{T-1}}(x^{T-1})}-1\big)$ is upper semi-analytic (note that $\mathcal S$ is a Borel set, so the domain restriction does not cause any issue). In particular,  $u$ is also upper semi-analytic on $\mathcal S$ (see \Cref{lemma:supinf} in \ref{app:lusin}). By an analogous argument, $l$ is lower semi-analytic on $\mathcal S$. Moreover, $$(\mu-1)^{-1}\leq u(x^{T-1})\leq\tilde\lambda_1^{x^{T-1}}\leq l(x^{T-1})\leq\mu^{-1}$$ for all $x^{T-1}\in\mathcal S$, and so in particular by \Cref{prop:ufl} (see \ref{app:lusin}) there is a Borel function $\lambda_{T}:\mathcal S\to I_\mu$ such that $l(x^{T-1})\geq \lambda_{T}(x^{T-1}) \geq u(x^{T-1})$ for all $x^{T-1}\in\mathcal S$. We can extend $\lambda_{T}$ to the whole $\X^{T-1}$, by setting $\lambda_{T}(x^{T-1}) = 0$, if $x^{T-1}\in \X^{T-1}\setminus \mathcal S$. Noting that $x^{T-1}\in\X^{T-1}\setminus \mathcal S$ implies that $E_{T-1+t}(x^{T-1}, y^t)=0$ for any $t\geq 1$ and $y^t\in\X^{t}$,\footnote{Indeed, for every $t$ the non-negativity of $E_{T-1+t}$ implies that this must be true for $Q$-almost every $y^t$, for every $Q\in\bar\Hh^\infty_\mu$, and so for every $y^t\in\X^t$.} one can easily verify that $\lambda^{T}$ dominates $E$ at level $T$. 

So, we can construct iteratively a sequence $\lambda^\infty\in\Lambda_\mu^\infty$ such that, for each $T\geq 1$, $\lambda^T$ dominates $E$ at level $T$. It follows immediately that $E^{\lambda^\infty}\succeq E$, and so $\Ecbinf_\mu$ is a majorising e-process class for $\bar\Hh_\mu^\infty$. \qed
\end{proof}

\subsubsection[Proof of Lemma 4]{Proof of \Cref{lemma:trees}}\label{app:trees}
\begin{figure}[t!]
  \centering
  \begin{subfigure}[t]{0.45\textwidth}
    \centering
    {\raisebox{7.5mm}{\begin{tikzpicture}[
  grow=right,
  sibling distance=8mm,
  level distance=20mm,
  edge from parent/.style={draw}
]
  \node {$\bullet$}
    child { node {$a_2$}
      child { node {$a_2, b_4$}
        child [missing]
        child [missing]
      }
      child [missing]
      child { node {$a_2, b_3$}
        child { node {$a_2, b_3, c_6$} }
        child { node {$a_2, b_3, c_5$} }
      }
    }
    child [missing]
    child [missing]
    child [missing]
    child { node {$a_1$} 
    };
\end{tikzpicture}}}
    \caption{The pruned subtree generated by the stopping time $\tau^\star$ (see main text) acting on the tree in \Cref{fig:tree} (for $d=(a^2, b^4, c^8)$)}\label{fig:prune}
  \end{subfigure}
  \hfill
  \begin{subfigure}[t]{0.45\textwidth}
    \centering
    \begin{tikzpicture}[
  grow=right,
  sibling distance=8mm,
  level distance=20mm,
  edge from parent/.style={draw},
  every node/.style={inner sep=1pt},
  circ/.style={circle, draw, minimum size=4mm, inner sep=0pt}
]
  \node {$\bullet$}
    child { node {$\bullet$}
      child { node[circ] {}
        child { node {$\bullet$} }
        child { node {$\bullet$} }
      }
      child [missing]
      child { node {$\bullet$}
        child { node[circ] {} }
        child { node[circ] {} }
      }
    }
    child [missing]
    child [missing]
    child [missing]
    child { node[circ] {}
      child { node {$\bullet$}
        child { node {$\bullet$} }
        child { node {$\bullet$} }
      }
      child [missing]
      child { node {$\bullet$}
        child { node {$\bullet$} }
        child { node {$\bullet$} }
      }
    };
\end{tikzpicture}
    \caption{The mask that can be applied to the tree in \Cref{fig:tree} to obtain the same pruned tree as generated by $\tau^\star$ (\Cref{fig:prune}).}\label{fig:mask}
  \end{subfigure}

  \caption{Pruned tree and mask representing the action on $d\in\D_T$ of a stopping time $\tau^\star$ bounded by $T$.}
\end{figure}
We recall that for any time horizon $T \geq 1$, $\D_T$ is the set of admissible tuples of coefficients that generate elements of $\bar\Hh_\mu^T$ via the construction that we outlined in \ref{app:coarse}. We now generalise the idea behind \Cref{fig:tree} and associate to $d\in\D_T$ a binary tree structure, with one root and $T$ additional levels. This will be particularly useful to deal with e-processes and stopping time. For convenience, in what follows we denote as $\T_T$ the set of all stopping times bounded by $T$. We can notice that a stopping time $\tau\in\T_T$ defines a subset of the tree. More precisely, it defines a pruned tree, namely a fully connected subtree that contains the root. Note that the converse is also true. Given $d\in\D_T$ and a pruned tree, there is a $\tau\in\T_T$ that induces this subtree. To make this more explicit, consider the case $T=3$ again. We might consider a stopping time $\tau^\star$ that stops at $1$ if $a_1$ is observed. Conversely, if $a_2$ is present, it if $b_4$ is observed stops at $2$, otherwise at $3$. In such case, the observable states are the leaves of the tree in \Cref{fig:prune}. Conversely, such pruned tree is induced by any stopping time $\tau$ that behaves like $\tau^\star$ on $d$.

We remark that fixed a $d\in\D_T$, although there are infinitely many stopping times bounded by $T$ (at least if $\X$ has infinite cardinality), they result on finitely many possible pruned trees when applied to $Q_d$. In particular, given a vector of Borel functions $(f_t)_{t\in[0:T]}$, with $f_t:\X^t\to\R$, and fixed $d\in\D_T$, we have that the set $\{E_{Q_d}[f_\tau]\,:\,\tau\in\T_T\}$ has finitely many elements. Alternatively, given $d$, each $\tau$ can be represented as a \emph{mask} applied on the full tree, in a way that defines the pruned tree structure. By mask, here we mean something like what depicted in \Cref{fig:mask}, where the blank circles represent the leaves of the pruned subtree (\Cref{fig:prune}) induced by the stopping time $\tau^\star$ when applied to the full tree of \Cref{fig:tree}. We denote as $\mathcal M_T$ the (finite) set of masks for for the binary tree with the root and $T$ levels. Of course, the same stopping time $\tau$ can be associated to different masks when applied to different $d\in\D_T$, as the times at which $\tau$ stops depend on the value of the realisations of $X_t$ observed, namely on $d$. However, this will not prevent us from using the fact that there are only finitely many masks in $\mathcal M_T$, which will be a main ingredient in the proof of \Cref{lemma:trees}. In practice, the key observation that we need is the fact that for any pair $(Q,\tau)\in\bar\Hh_\mu^T\times\T_T$, we can find a pair $(d,M)\in\D_T\times\mathcal M_T$ that define the very same pruned tree. The converse is also true, for any pair $(d,M)$ we can find a $(Q,\tau)$ that generated the same pruned tree. We also note that for any vector $(f_t)_{t\in[0:T]}$ of Borel functions $f_t:\X^t\to\R$, for any 
$(Q,\tau)\in\bar\Hh_\mu^\infty\times\T_T$, the value of $\E_Q[f_\tau]$ if fully determined by the pruned tree generated by the pair $(Q,\tau)$. For instance, if we consider a pair $(Q,\tau)$ that induces the pruned tree of \Cref{fig:prune}, recalling the definition \eqref{eq:W} of $W$, one can easily work out that \begin{align}\begin{split}\label{eq:masked}\E_Q[f_\tau] &= W(a_1,a_2)f_1(a_1) + (1-W(a_1,a_2))(1-W(b_3,b_4))f_2(a_2,b_4) \\&+ (1-W(a_1,a_2))W(b_3,b_4)\big(W(c_5,c_6)f_3(a_2,b_4,c_5) + (1-W(c_5,c_6))f_3(a_2,b_4,c_5)\big)\,,\end{split}\end{align} no matter the specific $Q$ and $\tau$ involved. From this observation, we see that we can well define a mapping $h$ that, given $d\in\D_T$, $M\in\mathcal M_T$, and a vector $f=(f_t)_{t\in[0:T]}$ of Borel functions, returns the value $$h(d,M,f) = \E_Q[f_\tau]\,,$$ where $(Q,\tau)\in\bar\Hh_\mu^T\times\T_T$ is any pair that generates the same pruned tree as $(d,M)$.

After all these preliminaries, we are finally ready to prove \Cref{lemma:trees}.

\begin{proof}[of \Cref{lemma:trees}]
We show that the claim holds for $u$, the proof for $l$ being analogous. Let $f$ be a sequence of bounded Borel functions as in the statement. Fix $t\geq 1$. For $x^t\in\X^t$ recall that, for all $s\geq 1$, we let $f^{x^t}_s:y^s\mapsto f_{t+s}(x^t,y^s)$. Fix $T\geq 1$. Let $u_T:\X^t\to\R$ be defined as
$$u_T(x^t) = \sup_{\tau\in\T_T}\sup_{Q\in\bar\Hh_\mu^T}\E_Q[f^{x^t}_\tau]\,.$$
Following the discussion above, we can also write
$$u_T(x^t) = \max_{M\in\mathcal M_T}\sup_{d\in\D_T}h(d, M, f^{x^t})\,,$$ where we can take the maximum as $\mathcal M_T$ has finite cardinality. Now, note that for each $M\in\mathcal M_T$, the mapping $(d,x^t)\mapsto h(d,M, f^{x^t})$ is Borel. This follows from the fact that $W$, defined in \eqref{eq:W}, and all the $f_{t+s}$ are Borel.\footnote{See \eqref{eq:masked} (of course replacing $f$ with $f^{x^t}$) to get a more concrete idea of how this mapping looks like when $T=3$, with $M$  the mask in \Cref{fig:mask} and $d$  expressed as $(a^2, b^4, c^8)$.} In particular, by \Cref{lemma:supinf}, for each $M\in\mathcal M_T$ the mapping $x^t\mapsto \sup_{d\in\D_T}h(d, M, f^{x^t})$ is upper semi-analytic. Since the maximum of finitely many upper semi-analytic functions is upper semi-analytic, we conclude that $u_T$ is upper semi-analytic. We can then notice that $u = \sup_{T\geq 1}u_T$. Since this is the countable supremum of upper semi-analytic functions, it is upper semi-analytic. \qed
\end{proof}

\subsection{A functional variant of Lusin's separation theorem}\label{app:lusin}
First, let us recall a few standard definitions and results from descriptive set theory. We refer to \cite{karoui2013capacities} or to the monograph \cite{kechris1995classical} for more details.
\begin{definition}
	A set $A\subseteq\R^d$ is called \emph{analytic} if there exists a Polish space $\mathcal Y$, and a Borel set $B\subseteq 
    \R^d\times \Y$. such that $A = \pi(B)$, where $\pi$ is the projection $(x^d,y)\mapsto x^d$, from $\R^d\times\Y$ to $\R^d$. A set $C\subseteq\R^d$ is called \emph{co-analytic} if it is the complement of an analytic set.
\end{definition}
\begin{theorem}[Lusin's separation theorem] \label{thm:lusin}
	Let $A$ be an analytic set in $\R^d$ and $C$ a co-analytic set in $\R^d$. If $A\subseteq C$, then there exists a Borel set $B$ such that $A\subseteq B\subseteq C$. 
\end{theorem}
\begin{definition}
	Let $\Z\subseteq\R^d$ be a Borel set and $f:\Z\to\overline\R$. $f$ is \emph{upper semi-analytic} if its superlevel sets are analytic (namely for every $r\in\R$ the sets $\{f\geq r\}$ and $\{f>r\}$ are analytic). $f$ is lower semi-analytic if its sublevel sets are analytic (namely for every $r\in\R$ the sets $\{f\leq r\}$ and $\{f<r\}$ are analytic). 
\end{definition}
Clearly, any Borel function is both upper and lower semi-analytic. 
\begin{lemma}\label{lemma:supinf}
	Let $\Z\subseteq\R^d$ and $\Z'\subseteq\R^{d'}$ be Borel sets. Let $f:\Z\times\Z'\to\overline\R$ be an upper semi-analytic function. Then $u:\Z\to\overline\R$ defined as $u(z) = \sup_{z'\in\Z'}f(z,z')$ is upper semi-analytic. Similarly, let $l:\Z\to\overline\R$ be given by $l(z) = \inf_{z'\in\Z'}f(z,z')$. Then, $l$ is lower semi-analytic.
\end{lemma}

We now prove a consequence (\Cref{prop:ufl}) of Lusin's separation theorem that, although might be already known, we could not find in the literature. In short, we want to prove that if a lower semi-analytic function dominates an upper semi-analytic function, then there is a Borel function that separates them. 
\begin{lemma}\label{lem:usa}
    Let $\Z\subseteq\R^d$ be a Borel set and $u:\Z\to\R$ be an upper semi-analytic function bounded from below. Then there exists a non-decreasing sequence of simple (namely taking finitely many values) upper semi-analytic functions $(u_n)_{n\geq 1}$ such that $u_n\to u$ point-wise. Similarly, if $l:\Z\to\R$ is a lower semi-analytic function bounded from above, there exists a non-increasing sequence of simple lower semi-analytic functions $(l_n)_{n\geq 1}$ such that $l_n\to l$ point-wise.
\end{lemma}
\begin{proof}
    First, assume that $u$ is bounded. Without loss of generality we can assume that $u$ takes values in $[0,1]$. For each $n$, for $t=0,\dots, 2^n$, let $A_t^n = \{u\geq t2^{-n}\}$. Each $A_t^n$ is an analytic set. Define $u_n$ as follows. For any $x\in A_{2^n}^n$, $u_n(x) = 1$. For any $t=0,\dots 2^n-1$, for $x\in A_t^n\setminus A_{t+1}^n$, $u_n(x) = t2^{-n}$. Then, it is easily checked that $u_n$ is upper semi-analytic. Moreover, $u_n$ takes finitely many values, and by construction it is non-decreasing and converges uniformly to $u$. 

    Now, let $u$ be only bounded from below. Without loss of generality we can assume that $u$ is non-negative. Then, for each integer $t\geq 1$ we can define $v_t = \min(u,t)$. Each $v_t$ is a bounded upper semi-analytic function, and in particular, we have that there is a non-decreasing sequence $(v_{t,n})_{n\geq 1}$ of simple upper semi-analytic functions that converges uniformly to $v_t$. Now, for each $t\geq 1$, there is $n_t\geq 1$ such that $\sup_{x\in\Z}|v_{t,n_t}(x)-v_t(x)|\leq 1/t$. We define $u_t = \max_{s\leq t} v_{s, n_s}$. Each $u_t$ is an upper semi-analytic simple function, as the maximum of finitely many upper semi-analytic simple functions. By construction, the sequence $(u_t)_{t\geq 1}$ is non-decreasing and $u_t\to u$ point-wise, since $v_t\to u$.

    The conclusion for $l$ follows automatically, as $-l$ is upper semi-analytic and bounded from below.\qed
\end{proof}

\begin{proposition}\label{prop:ufl}
    Let $\Z\subseteq\R^d$ be a Borel set. Let $u:\Z\to\R$ be an upper semi-analytic function bounded from below and $l:\Z\to\R$ a lower semi-analytic function bounded from above. If $u\preceq l$, there exists a Borel function $b:\Z\to\R$ such that $u\preceq b\preceq l$. 
\end{proposition}
\begin{proof}
    We start by considering the case of simple functions, namely we assume that there is a finite set $\Phi = \{\phi_1,\dots,\phi_N\}$ where $u$ and $l$ are valued. Without loss of generality we can assume that $\Phi$ is ordered increasingly (namely, $\phi_{i+1}>\phi_i$). For any $i$, we have that the set $U_i = \{u\geq \phi_i\}$ is analytic, while $L_i = \{l\geq\phi_i\}$ is co-analytic. The condition $u\preceq l$ implies that $U_i\subseteq L_i$. In particular, by Lusin's separation theorem (\Cref{thm:lusin}), there is a Borel set $B_i$ such that $U_i\subseteq B_i\subseteq L_i$. Let $D_N = B_N$. For $1\leq i<N$, define $D_i = B_{i}\setminus B_{i+1}$. Then, all these sets are Borel, and it is easy to check that the function $$f = \sum_{i=1}^N\phi_i \one{D_i}$$ is Borel and satisfies $u\preceq f\preceq l$. 

    Now that we have proved the desired claim for the case where $u$ and $l$ are simple, let us consider the generic case. Since $u$ is upper semi-analytic and bounded from below, by \Cref{lem:usa} there is a non-decreasing sequence $(u_n)_{n\geq 1}$, of simple upper semi-analytic functions, that converges point-wise to $u$. Similarly, there is a non-increasing sequence $(l_n)_{n\geq 1}$, of simple lower semi-analytic functions, that converges to $l$. For each $n$ we have $u_n\preceq u\preceq l\preceq l_n$, so there is a sequence $(f_n)_{n\geq 1}$ of Borel functions such that $u_n\preceq f_n\preceq l_n$ for all $n$. Let $f = \limsup_{n\to\infty} f_n$. $f$ is Borel and $u_n\preceq f\preceq l_n$ for all $n$. In particular, $u\preceq f\preceq l$, as desired.\qed 
\end{proof}

\end{document}